\documentclass[a4paper,11pt, twoside]{article}
\usepackage{amsmath,amsthm}
\usepackage{amssymb,latexsym}
\usepackage{mathrsfs}
\usepackage{enumerate}
\usepackage[colorlinks,citecolor=red]{hyperref}
\usepackage[numbers,sort&compress]{natbib}
\usepackage{indentfirst}
\usepackage{mathtools}
\usepackage{paralist,bbding,pifont}
\newtheorem{theorem}{Theorem}[section]
\newtheorem{corollary}{Corollary}[section]
\newtheorem{lemma}{Lemma}[section]
\newtheorem{proposition}{Proposition}[section]
\newtheorem{definition}{Definition}[section]
\newtheorem{remark}{Remark}[section]

\theoremstyle{definition} \theoremstyle{remark}
\numberwithin{equation}{section}
\allowdisplaybreaks

\begin{document}

\markboth{Y.Z. Yang, Y. Zhou et al}{Time-space fractional Schr\"{o}dinger equation on $\mathbb{R}^{d}$}

\date{}

\baselineskip 0.22in

\title{\bf Local/global well-posedness analysis of time-space fractional Schr\"{o}dinger equation on $\mathbb{R}^{d}$}

\author{ Yong Zhen Yang$^1$, Yong Zhou$^{1,2}$\\[1.8mm]
\footnotesize {Correspondence: yozhou@must.edu.mo}\\
\footnotesize  {$^{1}$ Faculty of Mathematics and Computational Science, Xiangtan University}\\
\footnotesize  {Hunan 411105, P.R. China}\\[1.5mm]
\footnotesize {$^2$ Macao Centre for Mathematical Sciences, Macau University of Science and Technology}\\
\footnotesize {Macau 999078, P.R. China}\\[1.5mm]
}

\maketitle

\begin{abstract}
We investigate a class of nonlinear time-space fractional Schr\"{o}dinger equations with nonlocal effects in both time and space. The time derivative is of Achar type, and the space operator is a $\phi(-\Delta)$-type operator defined via a Bernstein function $\phi$. This nonlocality invalidates classical Strichartz estimates. By combining asymptotic analysis of Mittag-Leffler functions, the H\"{o}rmander multiplier theorem, and harmonic analysis techniques, we establish a Gagliardo-Nirenberg inequality in $\phi$-Triebel-Lizorkin spaces and derive key Sobolev estimates for the solution operator. These analyses yield the local and global well-posedness of the equations in appropriate Banach spaces. Our work demonstrates the effectiveness of the $\phi(-\Delta)$-framework for handling fractional dispersive equations with nonlocality.\\ [2mm]
{\bf MSC:} 26A33; 35R11.\\
{\bf Keywords:}time-space fractional Schr\"{o}dinger equations; global/local well-posedness; Sobolev estimates
\end{abstract}

\baselineskip 0.25in

\section{Introduction}

Over the past few decades, fractional calculus has been widely applied in various fields, including fluid mechanics, physics, materials science, signal processing and system identification, automatic control and robotics, electrochemistry, biology, and economics. Theoretical and numerical experiments have demonstrated that, compared to traditional derivatives, fractional derivatives, due to their non-local effects, are particularly suitable for characterizing physical processes in the real world with memory, hereditary, and global correlations. They offer advantages such as clear physical meaning of parameters and accurate descriptions of complex systems. For instance, fractional partial differential equations are used to describe viscoelastic mechanical models, electrical conduction in biological systems, electromagnetic wave propagation, fractional electrical circuits, fractional control systems and fractional controllers, and fluid flow in fractal media. For more details, see \cite{Kilbas,Zhou2014,Zhou2024,Podlubny}.

The Schr\"{o}dinger equation is one of the cornerstones of quantum mechanics, providing a mathematical framework for understanding the properties of atoms, molecules, and other microscopic particles. The time-fractional Schr\"{o}dinger equation was first introduced by Naber \cite{Naber}, who employed a core rotation technique to modify the exponent of $i$ to $i^{\alpha}$, resulting in the following equation:
$$
i^{\alpha}\partial^{\alpha}_{t}w(t,x)=-\Delta w(t,x),
$$
In this work, Naber investigated the time-dependent Hamiltonian characteristics of the time-fractional Schr\"{o}dinger equation and provided a physical interpretation, demonstrating that the fractionalization of $i$ yields more intuitive physical insights. However, Narahari Achar \cite{Achar} reached a contrasting conclusion, advocating for retaining the standard $i$:
$$
i\partial^{\alpha}_{t}w(t,x)=-\Delta w(t,x).
$$
Moreover,  Laskin \cite{Laskin} derived the space-fractional Schr\"{o}dinger equation by replacing Brownian paths with L\'{e}vy stable paths in the Feynman path integral framework. Notably, the Markov property of the solution remains preserved in this formulation. For space fractional operators, a prominent example is the fractional Laplacian $(-\Delta)^{\frac{\beta}{2}}$,$0<\beta<2$, which holds significant research value in both mathematics and finance. Professor L. Caffarelli, a member of the U.S. National Academy of Sciences and Wolf Prize laureate, made groundbreaking contributions in \cite[\emph{Ann. Math.}, \emph{Invent. Math.}]{Caffarelli,Caffarelli1} by using the unique property of $(-\Delta)^{\frac{\beta}{2}}$ that maps Dirichlet boundary conditions to Neumann-type boundary conditions. His work resolved critical issues such as maximal regularity and free boundary regularity for fractional Laplacian equations, garnering widespread acclaim in the academic community. Consequently, in recent years, studies on the nonlinear time-space fractional Schr\"{o}dinger equation which incorporates both time and space fractional derivatives have attracted considerable attention. Key research directions include well-posedness, dispersion estimates, and decay properties, as evidenced by works such as \cite{Banquet,Grande,Lee,Su,Su1}.

In this paper, We study the space-time fractional Schr\"{o}dinger equation governed by the operator \( \phi(-\Delta) \):
\begin{align}\label{E.Q.T.F.S.E}
\begin{cases}
i\partial_t^\alpha w = \phi(-\Delta)w + g(w) & \text{in } (0,\infty) \times \mathbb{R}^d, \\
w(0,\cdot) = w_0 & \text{in } \mathbb{R}^d,
\end{cases}
\end{align}
where \( \partial_t^\alpha \) denotes the Caputo fractional derivative (\( 0 < \alpha < 1 \)) and \( g \) is a prescribed nonlinearity, and $\phi$ is a Bernstein function satisfying $\phi(0^{+}) = 0$, that is, $\phi: (0,\infty) \rightarrow (0,\infty)$ and
\[
(-1)^{k+1}\phi^{(k)}(x) \geq 0, \quad x > 0, \quad k = 0, 1, 2, \ldots
\]
The operator $\phi(\Delta):=-\phi(-\Delta)$ is the quantization of rotationally invariant subordinate Brownian motion with characteristic exponent $\phi(|\xi|^{2})$, and it can be defined as
\[
\phi(\Delta)w(x) = \mathcal{F}^{-1}\left(-\phi(|\xi|^{2})\mathcal{F}w(\xi)\right)(x), \quad w \in \mathcal{S}(\mathbb{R}^{d}).
\]
In particular, when $\phi(\lambda) = \lambda^{\frac{\beta}{2}}$ with $\beta \in (0,2)$, the operator $\phi(-\Delta)$  reduces to the fractional Laplacian $(-\Delta)^{\frac{\beta}{2}}$.In this case, equation \eqref{E.Q.T.F.S.E}becomes:
\begin{align}\label{E.Q.T.F.S.E.1}
\begin{cases}
    i\partial_{t}^{\alpha}w(t,x) = (-\Delta)^{\frac{\beta}{2}}w(t,x) + g(w(t,x)) & \text{in } (0,\infty) \times \mathbb{R}^{d}, \\
    w(0,x) = w_{0}(x) & \text{in } \mathbb{R}^{d},
\end{cases}
\end{align}
Su \cite{Su} study the local well-posedness of equation \eqref{E.Q.T.F.S.E.1} in a suitable Banach space.Notably, due to the lack of compactness in the operator $\exp(-it(-\Delta)^{\frac{\beta}{2}})$and the absence of decay in its symbol, the subordination principle \cite{Zhou2014} cannot be directly applied to establish the compactness of the solution operator.Furthermore, $L^{p}-L^{r}$ estimates analogous to those for fractional diffusion equations \cite{Zhou2014,Zhou2024} are unavailable. Therefore, the study of the time-space fractional Schr\"{o}dinger equation using the subordination principle has certain limitations. To circumvent these limitations, Su \cite{Su} employed the Mellin transform and asymptotic analysis of the H-Fox function---defined via a Mellin-Barnes integral---to characterize solution operator properties. This approach enabled the derivation of decay estimates similar to those for fractional semilinear heat equations, thereby establishing the local well-posedness of mild solutions in an appropriate Banach space.

However, Su \cite{Su} method heavily relies on the homogeneity of the symbol associated with the operator. Specifically, it depends on the Fourier transform of homogeneous distributions \cite{Grafakos}, that is
$$
\mathcal{F}(|x|^{z})=\frac{\Gamma(\frac{z+d}{2})}{\Gamma(\frac{-z}{2})}|\xi|^{-z-d},
$$
which fails for operators with non-homogeneous symbols, such as general $\phi(-\Delta)$. Thus, extending the techniques used for equation \eqref{E.Q.T.F.S.E.1} to equation \eqref{E.Q.T.F.S.E} requires novel methodologies. To address this challenge, we first construct embedding theorems for $\phi$-type Besov space and $\phi$-type Triebel-Lizorkin space, building upon the framework introduced by Mikulevi\v{c}ius \cite[\emph{Potential Anal.}]{Mikulevicius}. Subsequently, we establish a Gagliardo-Nirenberg inequality in these $\phi$-type Triebel-Lizorkin space, which plays a pivotal role in deriving $L^{p}-L^{r}$ estimates for the solution operator. Next, by using the asymptotic properties of the Mittag-Leffler function (Proposition \ref{X.Z.of Mitag function}) as derived by Gorenflo \cite{Gorenflo}, combined with tools from harmonic analysis, including the H\"{o}rmander multiplier theorem, Gagliardo-Nirenberg inequality, and real interpolation techniques, we obtain $L^{p}-L^{r}$ estimates for the solution operator. This allows us to establish the local/global well-posedness and asymptotic behavior of mild solutions to E.q. \eqref{E.Q.T.F.S.E}. Crucially, our approach does not require homogeneity in the symbol of the spatial operator, thereby complementing the results of Su \cite{Su} and offering a distinct methodology.

This paper is organized as follows. In Section 2, we introduce some of the notations required for this paper, including fractional derivatives, the $\phi(\Delta)$-type operator and $\phi$-Besov space, $\phi$-Triebel-Lizorkin space etc. In Section 3, we construct the the Gagliardo-Nirenberg inequality in the $\phi$-Triebel-Lizorkin space and obtain some Sobolev estimates for the solution operator. In Section 4, we establish the global/local well-posedness of the E.q.\eqref{E.Q.T.F.S.E} and analyze the asymptotic behavior.
\section{Preliminaries}
In this section, we introduce definitions, notations, and key lemmas used throughout the paper.

We denote by \( C \) a generic constant that may vary from line to line in the formulas. The symbols \( a \vee b \) and \( a \wedge b \) represent the maximum and minimum of \( a \) and \( b \), respectively. We write \( a \lesssim b \) if \( a \leq Cb \) for some constant \( C > 0 \), and \( a \sim b \) if there exist positive constants \( C_1, C_2 \) such that \( C_1 b \leq a \leq C_2 b \).
Let $\mathcal{S}(\mathbb{R}^d)$ denote the Schwartz space of rapidly decaying smooth functions on $\mathbb{R}^d$, with its dual space $\mathcal{S}'(\mathbb{R}^d)$ consisting of tempered distributions. We further define the restricted Schwartz space
$$
\mathcal{\dot{S}}(\mathbb{R}^{d}) = \left\{ f \in \mathcal{S} : \partial^{\gamma}f(0) = 0 \text{ for all multi-indices } \gamma \in \mathbb{N}^d \right\},
$$
whose dual space corresponds to the quotient space $\mathcal{S}'/\mathcal{P}$, where $\mathcal{P}$ denotes polynomial functions.

The integral transforms are defined as follows: For any $f \in \mathcal{S}(\mathbb{R}^d)$, the Fourier transform pair is given by
\[
\mathcal{F}f(\xi) = \int_{\mathbb{R}^d} e^{-i x \cdot \xi} f(x) dx, \quad
\mathcal{F}^{-1}f(x) = \frac{1}{(2\pi)^{d/2}} \int_{\mathbb{R}^d} e^{i x \cdot \xi} f(\xi) d\xi,
\]
with $\mathcal{L}$ and $\mathcal{L}^{-1}$ denoting the Laplace transform and its inverse respectively.
By using the dual method, we can extend the $\mathcal{F}$ and $\mathcal{F}^{-1}$ to the $\mathcal{S}'$, that is for $f\in\mathcal{S}'$
$$
\langle\mathcal{F}f,g\rangle=\langle f,\mathcal{F}g\rangle,\quad\langle\mathcal{F}^{-1}f,g\rangle=\langle f,\mathcal{F}^{-1}g\rangle
\text{ for any }g\in\mathcal{S}(\mathbb{R}^d).
$$
We denote that $*$ and $\star$ represent the convolution to the variable of $t$ and $x$, respectively. Moreover,
for a measurable function $f \in \mathcal{M}(\mathbb{R}^d)$ with polynomially bounded growth at infinity, the operator $f(D)$ (where $D = i\partial_x$) is defined by Fourier duality as:
\[
f(D)a \, := \, \mathcal{F}^{-1}\big[\,f(\xi) \cdot (\mathcal{F}a)(\xi)\,\big].
\]

\begin{definition}
For a function $g\in L^{1}\big(0,\infty;\mathcal{S}(\mathbb{R}^{d})\big)$, $0<\alpha<1$, the Riemann-Liouville fractional integral and Caputo fractional derivative of $g$ can be defined that
\begin{align*}
_{0}J_{t}^{\alpha}\!g(t,x)&=g_{\alpha}(t)*g(t,x),\\
\partial^{\alpha}_{t}g(t,x)&=\frac{d}{dt}\left(g_{1-\alpha}(t)*(g(t,x)-g(0,x))\right)
\end{align*}
where $g_{\alpha}(t)=t^{\alpha}/\Gamma(\alpha)$.
\end{definition}
Next, we introduce the Mittag-Leffler function $E_{\alpha,\beta}(z)$, which is defined as follow:
\begin{equation*}
E_{\alpha,\beta}(\rho) = \sum_{n=0}^{\infty} \frac{\rho^{n}}{\Gamma(\alpha n + \beta)}, \qquad \text{for }\alpha, \beta > 0, \text{ and }\rho \in \mathbb{C},
\end{equation*}

The Mittag-Leffler function has the following properities, can refer to \cite{Kilbas,Gorenflo}.

\begin{proposition}\label{X.Z.of Mitag function}
For the Mittag-Leffler function \( E_{\alpha,\beta}(z) \), the following properties hold:

\begin{enumerate}[\rm(i)]
   \item The following Laplace transform equation holds:
    \begin{equation}\label{M.T.F.Laplace}
    \int_{0}^{\infty} e^{-st} t^{\beta-1} E_{\alpha,\beta}(\pm a t^{\alpha}) \, dt = \frac{s^{\alpha-\beta}}{s^{\alpha} \mp a} \quad \text{for} \quad \Re(s) > 0, \, a \in \mathbb{C}, \, |s^{-\alpha} a| < 1.
    \end{equation}
    \item If \( 0 < \alpha < 1\), \( \beta < 1 + \alpha \), \( z \neq 0 \), and \( |\arg(z)|<\alpha\pi\), then
    \begin{equation}\label{asbe1}
    E_{\alpha,\beta}(z) = \frac{1}{\alpha} z^{\frac{1-\beta}{\alpha}} \exp\left(z^{\frac{1}{\alpha}}\right)
    + \int_{0}^{\infty} \frac{1}{\pi \alpha} r^{\frac{1-\beta}{\alpha}} \exp\left(-r^{\frac{1}{\alpha}}\right)
    \frac{r \sin(\pi(1-\beta)) - z \sin(\pi(1-\beta + \alpha))}{r^2 - 2rz \cos(\pi \alpha) + z^2} \, dr.
    \end{equation}
    \item If $0<\alpha<1$, $\beta<1+\alpha$, $z \neq 0$, and $|\arg(z)|>\alpha\pi$, then
    \begin{align}\label{asbe2}
    E_{\alpha,\beta}(z) = \int_{0}^{\infty} \frac{1}{\pi \alpha} r^{\frac{1-\beta}{\alpha}} \exp\left(-r^{\frac{1}{\alpha}}\right)
    \frac{r \sin(\pi(1-\beta)) - z \sin(\pi(1-\beta + \alpha))}{r^2 - 2rz \cos(\pi \alpha) + z^2} \, dr.
    \end{align}
\end{enumerate}
\end{proposition}
\begin{remark}\label{take value E.q.}
In particular, for $z=-it^{\alpha}\phi(|\xi|^{2})$, we have $|\arg(z)|=\pi/2$. Thus we obtain that
\begin{align*}
E_{\alpha,1}(-it^{\alpha}\phi(|\xi|^{2}))&=\frac{i\sin(\alpha\pi)}{\alpha\pi}\int_{0}^{\infty}
\frac{e^{-r^{\frac{1}{\alpha}}}t^{\alpha}\phi(|\xi|^{2})}{r^{2}
+2it^{\alpha}\phi(|\xi|^{2})r\cos(\alpha\pi)-t^{2\alpha}(\phi(|\xi|^{2}))^{2}}dr,\text{ if }\alpha\in\big(0,\frac{1}{2}\big),\\
E_{\alpha,1}(-it^{\alpha}\phi(|\xi|^{2}))&=\frac{1}{\alpha}
\exp\big(\cos(\frac{\pi}{2\alpha})t\phi(|\xi|^{2})^{\frac{1}{\alpha}}
-i\sin(\frac{\pi}{2\alpha})t\phi(|\xi|^{2})^{\frac{1}{\alpha}}\big)\\
 &\quad+\frac{i\sin(\alpha\pi)}{\alpha\pi}\int_{0}^{\infty}
\frac{e^{-r^{\frac{1}{\alpha}}}t^{\alpha}\phi(|\xi|^{2})}{r^{2}
+2it^{\alpha}\phi(|\xi|^{2})r\cos(\alpha\pi)-t^{2\alpha}(\phi(|\xi|^{2}))^{2}}dr,\text{ if }\alpha\in\big(\frac{1}{2},1\big),
\end{align*}
\text{and}
\begin{align*}
E_{\alpha,\alpha}(-it^{\alpha}\phi(|\xi|^{2}))&=\frac{\sin((1-\alpha)\pi)}{\alpha\pi}\int_{0}^{\infty}
\frac{r^{\frac{1}{\alpha}}e^{-r^{\frac{1}{\alpha}}}}{r^{2}
+2it^{\alpha}\phi(|\xi|^{2})r\cos(\alpha\pi)-t^{2\alpha}(\phi(|\xi|^{2}))^{2}}dr,\text{ if }\alpha\in\big(0,\frac{1}{2}\big),\\
E_{\alpha,\alpha}(-it^{\alpha}\phi(|\xi|^{2}))&=\frac{1}{\alpha}\exp\big(-\frac{i\pi}{2}\frac{1-\alpha}{\alpha}\big)
t^{1-\alpha}(\phi(|\xi|^{2}))^{\frac{1-\alpha}{\alpha}}
\exp\big(e^{-i\frac{\pi}{2\alpha}}t\phi(|\xi|^{2})^{\frac{1}{\alpha}}\big)\\
 &\quad+\frac{\sin((1-\alpha)\pi)}{\alpha\pi}\int_{0}^{\infty}
\frac{r^{\frac{1}{\alpha}}e^{-r^{\frac{1}{\alpha}}}}{r^{2}
+2it^{\alpha}\phi(|\xi|^{2})r\cos(\alpha\pi)-t^{2\alpha}(\phi(|\xi|^{2}))^{2}}dr,\text{ if }\alpha\in\big(\frac{1}{2},1\big).
\end{align*}
\end{remark}

\begin{remark}
	In particular, for $\alpha=\frac{1}{2}$, the asymptotic integral expansions of the Mittag-Leffler functions 
	$E_{\alpha,1}(-it^{\alpha}\phi(|\xi|^{2}))$ and $E_{\alpha,\alpha}(-it^{\alpha}\phi(|\xi|^{2}))$ in Remark \ref{take value E.q.} 
	are no longer valid. Indeed, the integral expansion must then be interpreted in the sense of the Cauchy principal value, i.e.
	$$
	P.V.\int_{0}^{\infty}\frac{e^{-r^{2}}t^{\frac{1}{2}}\phi(|\xi|^{2})}{r^{2}-t(\phi(|\xi|^{2}))^{2}}\,dr=
	\lim_{\varepsilon\rightarrow 0^{+}}\int_{|r-t^{\frac{1}{2}}\phi(|\xi|^{2})|\geq \varepsilon}
	\frac{e^{-r^{2}}t^{\frac{1}{2}}\phi(|\xi|^{2})}{r^{2}-t(\phi(|\xi|^{2}))^{2}}\,dr.
	$$
	Here, following \cite{Kilbas}, we employ the error function $\operatorname{erf}(z)$ and the complementary error function $\operatorname{erfc}(z)$ to obtain the closed forms
	\begin{align*}
		E_{\frac{1}{2},1}(z)&=e^{z^{2}}\big(1+\operatorname{erf}(z)\big)=e^{z^{2}}\big(\operatorname{erfc}(-z)\big),\qquad z\in\mathbb{C},\\
		E_{\frac{1}{2},\frac{1}{2}}(z)&=\frac{1}{\sqrt{\pi}}+ze^{z^{2}}\big(1+\operatorname{erf}(z)\big)=\frac{1}{\sqrt{\pi}}+ze^{z^{2}}\big(\operatorname{erfc}(-z)\big),\qquad z\in\mathbb{C},
	\end{align*}
	where $\operatorname{erf}(z)$ and $\operatorname{erfc}(z)$ are defined by
	$$
	\operatorname{erf}(z)=\frac{2}{\sqrt{\pi}}\int_{0}^{z}e^{-t^{2}}\,dt,\qquad 
	\operatorname{erfc}(z)=1-\operatorname{erf}(z)=\frac{2}{\sqrt{\pi}}\int_{z}^{\infty}e^{-t^{2}}\,dt.
	$$
	
	Furthermore, for $z=ix$, we define the imaginary error function $\operatorname{erfi}(x)$,
	$$
	\operatorname{erfi}(x)=-i\operatorname{erf}(ix)=\frac{2}{\sqrt{\pi}}\int_{0}^{x}e^{t^{2}}\,dt.
	$$
	As $x\to\infty$, $\operatorname{erfi}(x)$ admits the asymptotic expansion
	$$
	\operatorname{erfi}(x)\sim\frac{e^{x^{2}}}{\sqrt{\pi}x}\sum_{k=0}^{\infty}\frac{(2k-1)!!}{(2x^{2})^{k}},\qquad x\to\infty.
	$$
	Retaining the leading terms gives
	$$
	\operatorname{erfi}(x)\sim\frac{e^{x^{2}}}{\sqrt{\pi}x}\left(1+\frac{1}{2x^{2}}+\frac{3}{4x^{3}}+\frac{15}{8x^{6}}+O\!\left(\frac{1}{x^{8}}\right)\right),\qquad x\to\infty.
	$$
	Moreover, $\operatorname{erfi}(x)$ also admits the closed form
	$$
	\operatorname{erfi}(x)=\frac{e^{x^{2}}}{x\sqrt{\pi}}+R(x),\qquad 
	R(x)=-\frac{1}{x\sqrt{\pi}}+\frac{1}{\sqrt{\pi}}\int_{0}^{x}\frac{e^{t^{2}}-1}{t^{2}}\,dt.
	$$
	By L'Hôpital's rule, we observe that
	$$
	\lim_{x\to\infty}\frac{x\int_{0}^{x}\frac{e^{t^{2}}-1}{t^{2}}\,dt}{e^{x^{2}}}
	=\lim_{x\to\infty}\frac{\int_{0}^{x}\frac{e^{t^{2}}-1}{t^{2}}\,dt+\frac{e^{x^{2}}-1}{x}}{2xe^{x^{2}}}
	=\lim_{x\to\infty}\frac{\frac{e^{x^{2}}-1}{x}}{2xe^{x^{2}}}
	=\lim_{x\to\infty}\frac{1}{2x^{2}}=0,
	$$
	which implies $\displaystyle\lim_{x\to\infty}xe^{-x^{2}}R(x)=0$.
	In particular, $\operatorname{erfi}(x)\approx\frac{e^{x^{2}}}{\sqrt{\pi}x}$ for $x\gg 1$.
	
	Therefore, for $\alpha=1/2$, setting $a=t^{\frac{1}{2}}\phi(|\xi|^{2})$, we obtain
	\begin{align*}
		E_{\frac{1}{2},1}(-ia)&=e^{-a^{2}}\operatorname{erfc}(ia)=e^{-a^{2}}\big(1-i\operatorname{erfi}(a)\big),\\
		E_{\frac{1}{2},\frac{1}{2}}(-ia)&=\frac{1}{\sqrt{\pi}}-iae^{-a^{2}}\big(1-i\operatorname{erfi}(a)\big)=\frac{1}{\sqrt{\pi}}-iae^{-a^{2}}-ae^{-a^{2}}\operatorname{erfi}(a).
	\end{align*}
\end{remark}

Next, we introduce some facts about the Bernstein function $\phi$ and the operator $\phi(-\Delta)$. For more details, refer to \cite{Kim1,Kim2}.

For a function $\phi: (0, \infty) \rightarrow (0, \infty)$ satisfying $\phi(0^+) = 0$, if there exists a constant $a \geq 0$ such that
$$
\phi(x) = a x + \int_{0}^{\infty} \left(1 - e^{-t x}\right) \mathcal{M}(dt),
$$
where the L\'{e}vy measure $\mathcal{M}$ satisfies
$$
\int_{0}^{\infty} (1 \wedge t) \mathcal{M}(dt) < \infty,
$$
then $\phi$ is called a Bernstein function. Let $\phi$ be a Bernstein function associated with a non-negative L\'{e}vy process $S_t$ (subordinator) on a probability space $(\Omega, \mathcal{Z}, \mathbb{P})$, characterized by its Laplace transform:
\[
\mathbb{E}[e^{-\lambda S_t}] = \exp(-t\phi(\lambda)), \quad \lambda \geq 0.
\]

Given a $d$-dimensional Brownian motion $X_t$ independent of $S_t$, we define the subordinate Brownian motion $W_t := X_{S_t}$. The infinitesimal generator of this process is given by the operator $\phi(\Delta)$, satisfying for $g \in \mathcal{S}(\mathbb{R}^d)$:
\[
\phi(\Delta)g(x) = \lim_{t\to 0^+} \frac{\mathbb{E}[g(x+W_t)] - g(x)}{t}.
\]
Moreover, the operator $\phi(\Delta)$ is an integro-differential operator defined as
$$
\phi(\Delta) g(x) = a \Delta g(x) + \int_{\mathbb{R}^d} \left(g(x + y) - g(x) - \nabla g(x) \cdot y I_{|y| \leq 1}\right) j(|y|) dy, \quad g \in \mathcal{S}(\mathbb{R}^d),
$$
where
$$
j(|y|) = \int_{(0, \infty)} \frac{1}{(4 \pi s)^{\frac{d}{2}}} \exp\left(-\frac{|y|^2}{4 s}\right) \mathcal{M}(ds),
$$
and it can also be defined as
$$
\phi(\Delta) g(x) = -\mathcal{F}^{-1}\left(\phi(|\xi|^2) \mathcal{F} g(\xi)\right)(x), \quad g \in \mathcal{S}(\mathbb{R}^d).
$$

For the Bernstein function $\phi$, it is easy to see that
\begin{align}\label{B.S.T,func.bdd}
|x^n \phi^{(n)}(x)| \leq a I_{n=1} + \int_{0}^{\infty} (t x)^n e^{-t x} \mathcal{M}(dt) \lesssim \phi(x).
\end{align}

Next, we present the following assumption for the Bernstein function, which is taken from \cite{Kim1,Kim2}.

\textbf{Assumption H1} There exists $\delta \in (0, 1]$ such that
\begin{align}\label{S.C.condition}
c \left(\frac{K}{k}\right)^{\delta} \leq \frac{\phi(K)}{\phi(k)}, \quad 0 < k < K < \infty.
\end{align}

The above assumption is reasonable, and there exist some Bernstein functions satisfying the condition:
\begin{enumerate}[\rm(1)]
  \item $\phi(x) = x^{\gamma}$, $\gamma \in (0, 1]$.
  \item $\phi(x) = x^{\gamma_1} + x^{\gamma_2}$, $0 < \gamma_1, \gamma_2 \leq 1$.
  \item $\phi(x) = \left(x + m^{\frac{1}{\gamma}}\right)^{\gamma} - m$, $\gamma \in (0, 1)$, $m > 0$.
  \item $\phi(x) = \frac{x}{\log(1 + x^{\frac{\beta}{2}})}$, $\beta \in (0, 2)$.
\end{enumerate}

Next, we introduce some Littlewood-Paley theory. For more details, please refer to \cite{Wang,Bahouri,Grafakos,Triebel}.

Let $\varphi \in \mathcal{S}$, and $\hat{\varphi}$ be supported in $\{\xi: \frac{1}{2} < |\xi| < 2\}$. Define
$$
\hat{\varphi}_{j}(\xi) = \varphi(2^{-j}\xi), \quad j = 0, \pm 1, \pm 2, \ldots,
$$
satisfying
$$
\sum_{j \in \mathbb{Z}} \varphi_{j}(\xi) = 1, \quad \xi \neq 0.
$$
Let $\hat{\chi} = 1 - \sum_{j=1}^{\infty} \hat{\varphi}_{j}(\xi)$, then $\text{supp}\, \hat{\chi} \subset B(0,1)$. We denote
$$
\Delta_{j}g(x) = \varphi_{j} \star g, \quad j = 0, \pm 1, \pm 2, \ldots.
$$

For $s \in \mathbb{R}$ and $1 \leq p, q \leq \infty$, we can define the standard Besov space $B^{s}_{p,q}(\mathbb{R}^{d})$ and Triebel-Lizorkin space $F^{s}_{p,q}(\mathbb{R}^{d})$, which are the closures of $\mathcal{S}$ under the following norms:
$$
\|w\|_{B^{s}_{p,q}} := \|\chi \star w\|_{p} + \left( \sum_{j=0}^{\infty} 2^{jsq} \| \Delta_{j}w \|^{q}_{p} \right)^{\frac{1}{q}},
$$
and
$$
\|w\|_{F^{s}_{p,q}} := \|\chi \star w\|_{p} + \left\| \left( \sum_{j=0}^{\infty} 2^{jsq} | \Delta_{j}w |^{q} \right)^{\frac{1}{q}} \right\|_{p}.
$$

The homogeneous Besov space $\dot{B}^{s}_{p,q}$ and Triebel-Lizorkin space $\dot{F}^{s}_{p,q}$ are defined similarly. Based on the above content, we can define the following $\phi$-type Besov space $B^{s,\phi}_{p,q}(\mathbb{R}^{d})$ and $\phi$-type Triebel-Lizorkin space $F^{s,\phi}_{p,q}(\mathbb{R}^{d})$, which have been studied in \cite{Mikulevicius,Sato}.

\begin{definition}
For $s \in \mathbb{R}$ and $1 \leq p, q \leq \infty$, the $\phi$-Besov space $B^{\phi,s}_{p,q}(\mathbb{R}^{d})$ and $\phi$-Triebel-Lizorkin space $F^{\phi,s}_{p,q}(\mathbb{R}^{d})$ can be defined as follows:
$$
B^{\phi,s}_{p,q}(\mathbb{R}^{d}) = \left\{ w \in \mathcal{S}'(\mathbb{R}^{d}) : \|w\|_{B^{\phi,s}_{p,q}} = \| \chi \star w \|_{p} + \left\| \left\{ \phi(2^{2j})^{\frac{s}{2}} \| \Delta_{j}w \|_{p} \right\}_{j \in \mathbb{Z}^{+}} \right\|_{l^{q}} \right\},
$$
and
$$
F^{\phi,s}_{p,q}(\mathbb{R}^{d}) = \left\{ w \in \mathcal{S}'(\mathbb{R}^{d}) : \|w\|_{F^{\phi,s}_{p,q}} = \| \chi \star w \|_{p} + \left\| \left\| \left\{ \phi(2^{2j})^{\frac{s}{2}} | \Delta_{j}w | \right\}_{j \in \mathbb{Z}^{+}} \right\|_{l^{q}} \right\|_{L^{p}} \right\}.
$$
\end{definition}

The definitions of homogeneous $\phi$-Besov space and homogeneous $\phi$-Triebel-Lizorkin space are similar, with the removal of $\| \chi \star w \|_{p}$, changing $j \in \mathbb{Z}^{+}$ to $j \in \mathbb{Z}$, and replacing $\mathcal{S}'$ with $\mathcal{S}' / \mathcal{P}$.

\begin{remark}\label{some remark of bernstein function} The following facts hold:
\begin{enumerate}[\rm(i)]
 \item In particular, if $\phi(x) = x$, then $B^{\phi,s}_{p,q}(\mathbb{R}^{d})$ and $F^{\phi,s}_{p,q}(\mathbb{R}^{d})$ coincide with the standard spaces $B^{s}_{p,q}(\mathbb{R}^{d})$ and $F^{s}_{p,q}(\mathbb{R}^{d})$. And for $1<p<\infty$, we also have $L^{p}(\mathbb{R}^{d})\sim F^{0}_{p,2}(\mathbb{R}^{d})\sim \dot{F}^{0}_{p,2}(\mathbb{R}^{d})\sim F^{\phi,0}_{p,2}(\mathbb{R}^{d})\sim \dot{F}^{\phi,0}_{p,2}(\mathbb{R}^{d})$, where $\sim$ indicates that the corresponding norms are equivalent.
 \item Since $\phi$ is a Bernstein function, it satisfies $\phi(2^{2j}) \leq 1 \vee 2^{2j}$ for any $j \in \mathbb{Z}$. This implies that
 $$
 \big\|w\big\|_{B^{\phi,s}_{p,q}} \lesssim \big\|w\big\|_{B^{s}_{p,q}} \quad \text{and} \quad \big\|w\big\|_{F^{\phi,s}_{p,q}} \lesssim \big\|w\big\|_{F^{s}_{p,q}}.
 $$
 \item It is easy to see that
 $$
 F^{\phi,s}_{p,2} \sim H^{s,\phi}_{p}(\mathbb{R}^{d}),
 $$
 where
 $$
 H^{s,\phi}_{p}(\mathbb{R}^{d}) = \bigg\{ w \in \mathcal{S}': \big(I - \phi(\Delta)\big)^{\frac{s}{2}} w \in L^{p}(\mathbb{R}^{d}) \bigg\}
 $$
 is the $\phi$-Bessel potential space. More details about $H^{s,\phi}_{p}$ can be found in {\rm\cite{Farkas}}.
 \item The $\phi$-Bessel potential $\big(I - \phi(\Delta)\big)^{\frac{\nu}{2}}$ is an isomorphism:
 $$
 B^{s,\phi}_{p,q}(\mathbb{R}^{d}) \rightarrow B^{s-\nu,\phi}_{p,q}(\mathbb{R}^{d}), \quad \text{and} \quad F^{s,\phi}_{p,q}(\mathbb{R}^{d}) \rightarrow F^{s-\nu,\phi}_{p,q}(\mathbb{R}^{d}),
 $$
 which can be found in \rm\cite{Mikulevicius,Kim1}.
\end{enumerate}
\end{remark}

In the following, we introduce H\"{o}rmander multiplier theory. For more details, refer to \cite{Grafakos,Triebel}.

As is well-known, for any bounded linear operator $T$ on $L^{p}$ that commutes with translation-invariant operators, there exists a unique tempered distribution $m_{T}$ such that
$$
T(f) = m_{T} \star f \quad \text{for any } f \in \mathcal{S}(\mathbb{R}^{d}).
$$
Let $\mathcal{U}_{p}$ denote the set of bounded linear operators on $L^{p}$ that commute with translation-invariant operators. Then, we define
$$
\big\|T\big\|_{L^{p}-L^{p}} = \big\|m\big\|_{\mathcal{U}_{p}}, \quad Tg = \mathcal{F}^{-1}\big(m \mathcal{F}g\big),
$$
where the tempered distribution $m$ is also called a H\"{o}rmander multiplier. The multiplier $m$ has good scaling properties, that is, for any $\lambda>0$,
\begin{align}\label{M.P.scailing}
\big\|m(\lambda \cdot)\big\|_{\mathcal{U}_{p}} = \big\|m\big\|_{\mathcal{U}_{p}}.
\end{align}
The following multiplier theorem can be found in \cite{Grafakos}.

\begin{theorem}\label{Multiplier theorem}
Let $m:\mathbb{R}^{d}\setminus \{0\}\rightarrow\mathbb{C}$ be a bounded function and for any multi-index $\beta$ with $|\beta| \leq \lceil\frac{d}{2}\rceil + 1$, satisfying
$$
\left( \int_{M < |\xi| < 2M} \big|D_{\xi}^{\beta} m(\xi)\big|^{2} d\xi \right)^{\frac{1}{2}} \lesssim M^{\frac{d}{2} - |\beta|}, \quad \text{for any } M > 0.
$$
Then, the operator $g \mapsto \mathcal{F}^{-1}(m \mathcal{F}g)$ is of strong type $(p,p)$ for $1 < p < \infty$ and weak type $(1,1)$.
\end{theorem}
\section{Some Lemmas}
In this section, we prove some Lemma that are crucial for this paper.
\begin{lemma}\label{find Mild solution}
If $w$ satisfies E.q. \eqref{E.Q.T.F.S.E}, then the solution $w(t,x)$ can be represented as:
\begin{align}\label{mild solution}
w(t,x) &= \mathcal{S}_{\alpha,\phi}(t)w_{0}(x) + \frac{1}{i}
\int_{0}^{t}(t-\tau)^{\alpha-1}\mathcal{P}_{\alpha,\phi}(t-\tau)g(w(\tau,x)) \, d\tau,
\end{align}
where in the distribution sence,
$$
\mathcal{S}_{\alpha,\phi}(t)=\mathcal{F}^{-1}\big(E_{\alpha,1}(-it^{\alpha}\phi(|\xi|^{2}))\big) \text{ and }\mathcal{P}_{\alpha,\phi}(t)=\mathcal{F}^{-1}\big(E_{\alpha,\alpha}(-it^{\alpha}\phi(|\xi|^{2}))\big).
$$
\end{lemma}

\begin{proof}
By applying the Fourier transform on E.q. \eqref{E.Q.T.F.S.E}, we obtain:
\begin{align*}
\begin{cases}
i\partial_{t}^{\alpha} \mathcal{F}w(t, \xi) - \phi(|\xi|^{2}) \mathcal{F}w(t, \xi) - \mathcal{F}(g(w))(t, \xi) = 0 & \text{in } (0,\infty) \times \mathbb{R}^{d}, \\
\mathcal{F}w(0, \xi) = \mathcal{F}(w_0)(\xi) & \text{in } \mathbb{R}^{d},
\end{cases}
\end{align*}
By applying the Laplace transform, we have:
\begin{align*}
\mathcal{L}(\widehat{w}(s, \xi)) = \frac{s^{\alpha-1}}{s^{\alpha} +i\phi(|\xi|^{2})}\widehat{w}_0(\xi) + \frac{-i}{s^{\alpha} +i\phi(|\xi|^{2})} \mathcal{L}(\mathcal{F}(g(w)))(s, \xi),
\end{align*}
Additionally, considering Proposition \ref{X.Z.of Mitag function} and \eqref{M.T.F.Laplace}, we derive that
\begin{align*}
\widehat{w}(t, \xi) = E_{\alpha,1}(-it^{\alpha} \phi(|\xi|^{2})) \widehat{w}_0(\xi) + \frac{1}{i} \int_{0}^{t} (t-\tau)^{\alpha-1} E_{\alpha,\alpha}(-i(t-\tau)^{\alpha}\phi(|\xi|^{2})) \mathcal{F}(g(w))(\tau, \xi) \, d\tau
\end{align*}
By applying the inverse Fourier transform, we obtain that
\begin{align*}
w(t, x)= \mathcal{S}_{\alpha,\phi}(t)w_{0}(x) + \frac{1}{i}
\int_{0}^{t}(t-\tau)^{\alpha-1}\mathcal{P}_{\alpha,\phi}(t-\tau)g(w(\tau,x)) \, d\tau.
\end{align*}
\end{proof}
\begin{definition}\label{D.f.mild solution}
Let \(X\) be a Banach space and let \(0 < T \leq \infty\). If a measurable function \(w: [0, T] \rightarrow X\) satisfies the integral E.q. \eqref{mild solution}, then we call \(u\) a mild solution of Eq. \eqref{E.Q.T.F.S.E}. In particular, if \(T = \infty\), then we denote \(w\) as a global mild solution of E.q. \eqref{E.Q.T.F.S.E}.
\end{definition}

We establish the embedding theorem for $\phi$-type Besov spaces and $\phi$-type Triebel-Lizorkin spaces, as well as the Gagliardo-Nirenberg inequality in the $\phi$-Triebel-Lizorkin space.

\begin{lemma}\label{embedding theorem}
Let $1 \leq p, p_{0}, q, q_{0} \leq \infty$ and $s, s_{0} \in \mathbb{R}$. If
$$
\delta s_{0} - \frac{d}{p_{0}} = \delta s - \frac{d}{p},
$$
then the following embedding relations hold:
\begin{align*}
&\dot{B}^{s_{0}, \phi}_{p_{0}, q_{0}}(\mathbb{R}^{d}) \hookrightarrow \dot{B}^{s, \phi}_{p, q}(\mathbb{R}^{d}), \quad
B^{s_{0}, \phi}_{p_{0}, q_{0}}(\mathbb{R}^{d}) \hookrightarrow B^{s, \phi}_{p, q}(\mathbb{R}^{d}), \quad \text{for any } p_{0} \leq p, q_{0} \leq q, \\
&\dot{F}^{s_{0}, \phi}_{p_{0}, q_{0}}(\mathbb{R}^{d}) \hookrightarrow \dot{F}^{s, \phi}_{p, q}(\mathbb{R}^{d}), \quad
F^{s_{0}, \phi}_{p_{0}, q_{0}}(\mathbb{R}^{d}) \hookrightarrow F^{s, \phi}_{p, q}(\mathbb{R}^{d}), \quad \text{for any } p_{0} < p.
\end{align*}
\end{lemma}

\begin{proof}
We only prove the case for the nonhomogeneous spaces; the homogeneous spaces can be handled similarly.

First, we prove $B^{s_{0}, \phi}_{p_{0}, q_{0}}(\mathbb{R}^{d}) \hookrightarrow B^{s, \phi}_{p, q}(\mathbb{R}^{d})$.

By the Bernstein inequality \cite[Chapter 1]{Bahouri},
$$
\big\|\Delta_{j}g\big\|_{p} \lesssim 2^{j\left(\frac{d}{p_{0}} - \frac{d}{p}\right)} \big\|\Delta_{j}g\big\|_{p_{0}}, \quad
\big\|\chi \star g\big\|_{p} \lesssim \big\|\chi \star g\big\|_{p_{0}}.
$$
Noting that $\delta s_{0} - \delta s = \frac{d}{p_{0}} - \frac{d}{p}$, we obtain
\begin{align*}
\big\|g\big\|_{B^{s, \phi}_{p, q}} &= \big\|\chi \star g\big\|_{p} + \left( \sum_{j=0}^{\infty} \phi(2^{2j})^{\frac{s q}{2}} \big\|\Delta_{j}g\big\|_{p}^{q} \right)^{\frac{1}{q}} \\
&\lesssim \big\|\chi \star g\big\|_{p_{0}} + \left( \sum_{j=0}^{\infty} \phi(2^{2j})^{\frac{s q}{2}} 2^{j q \left( \frac{d}{p} - \frac{d}{p_{0}} \right)} \big\|\Delta_{j}g\big\|_{p_{0}}^{q} \right)^{\frac{1}{q}} \\
&\lesssim \big\|\chi \star g\big\|_{p_{0}} + \left( \sum_{j=0}^{\infty} \phi(2^{2j})^{\frac{s q}{2} - \frac{s_{0} q}{2}} 2^{j q \left( \frac{d}{p} - \frac{d}{p_{0}} \right)} \phi(2^{2j})^{\frac{s_{0} q}{2}} \big\|\Delta_{j}g\big\|_{p_{0}}^{q} \right)^{\frac{1}{q}} \\
&\lesssim \big\|\chi \star g\big\|_{p_{0}} + \left( \sum_{j=0}^{\infty} \phi(2^{2j})^{\frac{s_{0} q_{0}}{2}} \big\|\Delta_{j}g\big\|_{p_{0}}^{q_{0}} \right)^{\frac{1}{q_{0}}} \\
&= \big\|g\big\|_{B^{s_{0}, \phi}_{p_{0}, q_{0}}},
\end{align*}
where we used \eqref{S.C.condition}, the fact that $l^{q_{0}} \hookrightarrow l^{q}$, and $s_{0} > s$. This implies $B^{s_{0}, \phi}_{p_{0}, q_{0}}(\mathbb{R}^{d}) \hookrightarrow B^{s, \phi}_{p, q}(\mathbb{R}^{d})$.

Next, we prove that for any $p_{0} < p$, $F^{s_{0}, \phi}_{p_{0}, q_{0}}(\mathbb{R}^{d}) \hookrightarrow F^{s, \phi}_{p, q}(\mathbb{R}^{d})$. Indeed, we only verify that $F^{s_{0}, \phi}_{p_{0}, \infty}(\mathbb{R}^{d}) \hookrightarrow F^{s, \phi}_{p, 1}(\mathbb{R}^{d})$. Without loss of generality, assume $\big\|g\big\|_{F^{s_{0}, \phi}_{p_{0}, \infty}} = 1$, and consider
\begin{align*}
\big\|g\big\|_{F^{s, \phi}_{p, 1}} = \big\|\chi \star g\big\|_{p} + \left\| \sum_{j=0}^{\infty} \phi(2^{2j})^{\frac{s}{2}} |\Delta_{j}g| \right\|_{p}.
\end{align*}
We obtain
\begin{align*}
\left\| \sum_{j=0}^{\infty} \phi(2^{2j})^{\frac{s}{2}} |\Delta_{j}g| \right\|_{p}^{p} &= p \int_{0}^{K} s^{p-1} \left| \left\{ x : \sum_{j=0}^{\infty} \phi(2^{2j})^{\frac{s}{2}} |\Delta_{j}g| > s \right\} \right| ds \\
&\quad + p \int_{K}^{\infty} s^{p-1} \left| \left\{ x : \sum_{j=0}^{\infty} \phi(2^{2j})^{\frac{s}{2}} |\Delta_{j}g| > s \right\} \right| ds \\
&= I + II,
\end{align*}
where $K \gg 1$ is a constant to be chosen later.

Note that $p_{0} < p$, $s_{0} > s$, and \eqref{M.P.scailing}, hence
$$
\sum_{j=0}^{\infty} \phi(2^{2j})^{\frac{s}{2}} |\Delta_{j}g| \leq \sum_{j=0}^{\infty} \phi(2^{2j})^{\frac{s - s_{0}}{2}} \sup_{j \geq 0} \phi(2^{2j})^{\frac{s_{0}}{2}} |\Delta_{j}g| \lesssim \sup_{j \geq 0} \phi(2^{2j})^{\frac{s_{0}}{2}} |\Delta_{j}g|.
$$
Thus, we have
\begin{align*}
I &\lesssim \int_{0}^{K} K^{p - p_{0}} s^{p_{0} - 1} \left| \left\{ x : \sup_{j \geq 0} \phi(2^{2j})^{\frac{s_{0}}{2}} |\Delta_{j}g| > c s \right\} \right| ds \\
&\lesssim \int_{0}^{c K} t^{p_{0} - 1} \left| \left\{ x : \sup_{j \geq 0} \phi(2^{2j})^{\frac{s_{0}}{2}} |\Delta_{j}g| > t \right\} \right| dt \\
&\lesssim \left\| \sup_{j \geq 0} \phi(2^{2j})^{\frac{s_{0}}{2}} |\Delta_{j}g| \right\|_{p_{0}}^{p_{0}} \lesssim 1,
\end{align*}
where the constant depends on $K$, but $K$ is fixed.

Moreover, by Bernstein's inequality and \eqref{M.P.scailing}, we have that
\begin{align*}
\phi(2^{2j})^{\frac{s}{2}}\big|\Delta_{j}g\big|&\lesssim\phi(2^{2j})^{\frac{s}{2}}2^{j\frac{d}{p_{0}}}\big\|\Delta_{j}g\big\|_{p_{0}}
\lesssim \phi(2^{2j})^{\frac{s-s_{0}}{2}}2^{j\frac{d}{p_{0}}}\big(\sup_{j\geq 0}\phi(2^{2j})^{\frac{s_{0}}{2}}\big\|\Delta_{j}g\big\|_{p_{0}}\big)\\
&\lesssim 2^{j\frac{d}{p}}\bigg\|\sup_{j\geq 0}\phi(2^{2j})^{\frac{s_{0}}{2}}\big|\Delta_{j}g\big|\bigg\|_{p_{0}},
\end{align*}
Hence for any $A\in\mathbb{N}$,
\begin{align*}
\sum_{j=0}^{A-1}\phi(2^{2j})^{\frac{s}{2}}\big|\Delta_{j}g\big|\lesssim\sum_{j=0}^{A-1}2^{j\frac{d}{p}}\bigg\|\sup_{j\geq 0}\phi(2^{2j})^{\frac{s_{0}}{2}}\big|\Delta_{j}g\big|\bigg\|_{p_{0}}\lesssim 2^{\frac{Ad}{p}}.
\end{align*}
Note that $s>K\gg 1$, hence we can choose $A$ to be the largest natural number such that $2^{A\frac{d}{p}}\lesssim s/2$, i.e., $2^{A}\sim s^{\frac{p}{d}}$. Hence we get that
\begin{align*}
\frac{s}{2}<\sum_{j=A}^{\infty}\phi(2^{2j})^{\frac{s}{2}}\big|\Delta_{j}g\big|
\lesssim\sum_{j=A}^{\infty}\phi(2^{2j})^{\frac{s-s_{0}}{2}}\bigg(\sup_{j\geq 0}\phi(2^{2j})^{\frac{s_{0}}{2}}\big|\Delta_{j}g\big|\bigg)\lesssim 2^{A\delta(s-s_{0})}\bigg(\sup_{j\geq 0}\phi(2^{2j})^{\frac{s_{0}}{2}}\big|\Delta_{j}g\big|\bigg).
\end{align*}
Thus we obtain that
\begin{align*}
II&\lesssim\int_{K}^{\infty}s^{p-1}\bigg|\bigg\{x:\sum_{j=0}^{A-1}\phi(2^{2j})^{\frac{s}{2}}|\Delta_{j}g|>\frac{s}{2}\bigg\}
\bigg|ds+\int_{K}^{\infty}s^{p-1}\bigg|\bigg\{x:\sum_{j=A}^{\infty}\phi(2^{2j})^{\frac{s}{2}}|\Delta_{j}g|>\frac{s}{2}\bigg\}
\bigg|ds\\
&\lesssim\int_{K}^{\infty}s^{p-1}\bigg|\bigg\{x:C2^{\frac{Ad}{p}}>\frac{s}{2}\bigg\}
\bigg|ds+\int_{K}^{\infty}s^{p-1}\bigg|\bigg\{x:C2^{A\delta(s-s_{0})}\bigg(\sup_{j\geq 0}\phi(2^{2j})^{\frac{s_{0}}{2}}\big|\Delta_{j}g\big|\bigg)>\frac{s}{2}\bigg\}
\bigg|ds\\
&\lesssim\int_{K}^{\infty}s^{p-1}\bigg|\bigg\{x:C2^{A(\frac{d}{p}-\frac{d}{p_{0}})}\bigg(\sup_{j\geq 0}\phi(2^{2j})^{\frac{s_{0}}{2}}\big|\Delta_{j}g\big|\bigg)>\frac{s}{2}\bigg\}
\bigg|ds\\
&\lesssim\int_{K}^{\infty}s^{p-1}\bigg|\bigg\{x:\sup_{j\geq 0}\phi(2^{2j})^{\frac{s_{0}}{2}}\big|\Delta_{j}g\big|>cs^{\frac{p}{p_{0}}}\bigg\}
\bigg|ds\\
&\lesssim\int_{K'}^{\infty}t^{p_{0}-1}\bigg|\bigg\{x:\sup_{j\geq 0}\phi(2^{2j})^{\frac{s_{0}}{2}}\big|\Delta_{j}g\big|>t\bigg\}
\bigg|dt\\
&\lesssim\bigg\|\sup_{j\geq 0}\phi(2^{2j})^{\frac{s_{0}}{2}}\big|\Delta_{j}g\big|\bigg\|_{p_{0}}^{p_{0}}\lesssim 1.
\end{align*}
And combining $\big\|\chi\star g\big\|_{p}\lesssim\big\|\chi\star g\big\|_{p_{0}}$, we obtain that $F^{s_{0},\phi}_{p_{0},\infty}(\mathbb{R}^{d})\hookrightarrow F^{s,\phi}_{p,1}(\mathbb{R}^{d})$.
\end{proof}

\begin{lemma}\label{G.N.inequality}
Let $1\leq p,q,p_{0},p_{1}<\infty$ and $s,s_{0},s_{1}\in\mathbb{R}$. Then the following Gagliardo-Nirenberg inequality holds:
$$
\big\|g\big\|_{\dot{F}^{s,\phi}_{p,q}(\mathbb{R}^{d})}
\lesssim\big\|g\big\|^{\theta}_{\dot{F}^{s_{0},\phi}_{p_{0},\infty}(\mathbb{R}^{d})}
\big\|g\big\|^{1-\theta}_{\dot{F}^{s_{1},\phi}_{p_{1},\infty}(\mathbb{R}^{d})},\quad
\big\|g\big\|_{F^{s,\phi}_{p,q}(\mathbb{R}^{d})}\lesssim\big\|g\big\|^{\theta}_{F^{s_{0},\phi}_{p_{0},\infty}(\mathbb{R}^{d})}
\big\|g\big\|^{1-\theta}_{F^{s_{1},\phi}_{p_{1},\infty}(\mathbb{R}^{d})}
$$
if and only if
\begin{align*}
\delta s-\frac{d}{p}&=\bigg(\delta s_{0}-\frac{d}{p_{0}}\bigg)\theta+\bigg(\delta s_{1}-\frac{d}{p_{1}}\bigg)(1-\theta),\\
s&\leq \theta s_{0}+(1-\theta)s_{1},\\
\text{if }s&= \theta s_{0}+(1-\theta)s_{1},\text{ then }s_{0}\neq s_{1}.
\end{align*}
\end{lemma}

\begin{proof}
We only need to verify the inhomogeneous case, as the homogeneous case can be verified similarly.

Let $s^{*}=\theta s_{0}+(1-\theta)s_{1}$ and $\frac{d}{p^{*}}=\theta \frac{d}{p_{0}}+(1-\theta)\frac{d}{p_{1}}$. Then we have $s^{*}\geq s$ and $p\leq p^{*}$. Moreover,
\begin{align*}
\delta s^{*}-\frac{d}{p^{*}}=\delta s-\frac{d}{p}=\bigg(\delta s_{0}-\frac{d}{p_{0}}\bigg)\theta+\bigg(\delta s_{1}-\frac{d}{p_{1}}\bigg)(1-\theta).
\end{align*}
By combining the convexity of H\"{o}lder's inequality and Lemma \ref{embedding theorem}, we obtain
\begin{align*}
\big\|g\big\|_{F^{s,\phi}_{p,q}}\lesssim\big\|g\big\|_{F^{s^{*},\phi}_{p^{*},\infty}}
\lesssim\big\|g\big\|^{\theta}_{F^{s_{0},\phi}_{p_{0},\infty}(\mathbb{R}^{d})}
\big\|g\big\|^{1-\theta}_{F^{s_{1},\phi}_{p_{1},\infty}(\mathbb{R}^{d})}.
\end{align*}
This completes the proof.
\end{proof}

\begin{lemma}\label{M.L.Mittag-Leffer}
For $1 < p < \infty$, any $a \geq 0$, and $\sigma \in [0, 2]$, we have that
$$
\left(a + t^{\alpha} \phi(|\xi|^{2})\right)^{\frac{\sigma}{2}} E_{\alpha, 1}(-it^{\alpha} \phi(|\xi|^{2})) \in \mathcal{U}_{p}(\mathbb{R}^{d}).
$$
\end{lemma}

\begin{proof}
We first prove that
\begin{align}\label{t=1,M.L.condition}
\left(a + \phi(|\xi|^{2})\right)^{\frac{\sigma}{2}} E_{\alpha, 1}(-i\phi(|\xi|^{2})) \in \mathcal{U}_{p}.
\end{align}
By Theorem \ref{Multiplier theorem}, we need to verify that for any $M > 0$,
\begin{align}\label{multiplier condition}
\sum_{|\beta| \leq \lceil\frac{d}{2}\rceil + 1} \int_{\frac{M}{2} < |\xi| < 2M} \left|M^{|\beta|} D^{\beta}_{\xi} \left[\left(a + \phi(|\xi|^{2})\right)^{\frac{\sigma}{2}} E_{\alpha, 1}(-i\phi(|\xi|^{2}))\right]\right|^{2} d\xi \lesssim M^{d}.
\end{align}
By Fa\'{a} di Bruno's formula, for any multi-index $\beta$, we have
\begin{align*}
D^{\beta}_{\xi} \phi(|\xi|^{2}) = \sum_{\frac{|\beta|}{2} \leq j \leq 2|\beta|} \phi^{(j)}(|\xi|^{2}) \prod_{i=1}^{d} \xi_{i}^{\gamma_{i}}, \quad \sum_{i=1}^{d} \gamma_{i} = 2j - |\beta|.
\end{align*}
Thus, combining \eqref{B.S.T,func.bdd}, we obtain that
$$
\left|D^{\beta}_{\xi} \phi(|\xi|^{2})\right| \lesssim |\xi|^{-2j} \phi(|\xi|^{2}) |\xi|^{\sum_{i=1}^{d} \gamma_{i}} \lesssim |\xi|^{-|\beta|} \phi(|\xi|^{2}),
$$
which implies that
\begin{align*}
\left|D_{\xi}^{\beta} \left(a + \phi(|\xi|^{2})\right)^{\frac{\sigma}{2}}\right| &\lesssim \sum_{b_{1} + b_{2} + \cdots + b_{l} = \beta} \prod_{i=1}^{l} D_{\xi}^{b_{i}} \left(a + \phi(|\xi|^{2})\right) \left(a + \phi(|\xi|^{2})\right)^{\frac{\sigma}{2} - l} \\
&\lesssim |\xi|^{-|\beta|} \left(a + \phi(|\xi|^{2})\right)^{\frac{\sigma}{2}}.
\end{align*}
For $0 < M < 1$, note that
$$
E_{\alpha, 1}(-i\phi(|\xi|^{2})) = \int_{0}^{\infty} M_{\alpha, 1}(\theta) \exp(-i\theta\phi(|\xi|^{2})) d\theta,
$$
where $M_{\alpha}(\theta)$ is the Wright function, as referenced in \cite{Kilbas,Zhou2024}. Hence, we have
\begin{align*}
&\left|D^{\beta}_{\xi} \left[\left(a + \phi(|\xi|^{2})\right)^{\frac{\sigma}{2}} E_{\alpha, 1}(-i\phi(|\xi|^{2}))\right]\right| \\
&\lesssim \sum_{\substack{\beta_{1} + \beta_{2} = \beta, \\ \frac{|\beta_{2}|}{2} \leq l \leq |\beta_{2}|}} \left|D_{\xi}^{\beta_{1}} \left(a + \phi(|\xi|^{2})\right)^{\frac{\sigma}{2}} D_{\xi}^{\beta_{2}} E_{\alpha, 1}(-i\phi(|\xi|^{2}))\right| \\
&\lesssim \sum_{\substack{\beta_{1} + \beta_{2} = \beta, \\ \frac{|\beta_{2}|}{2} \leq l \leq |\beta_{2}|}} \left(a + \phi(|\xi|^{2})\right)^{\frac{\sigma}{2} + l} \left|\int_{0}^{\infty} M_{\alpha, 1}(\theta) \theta^{|\beta_{2}|} \exp(-i\theta\phi(|\xi|^{2})) d\theta\right| \\
&\leq C(d, \alpha, \delta, \beta) M^{-|\beta|},
\end{align*}
where we use the property of the Bernstein function $\phi(x) \leq 1 \vee x$ and the convergence of the Wright function, that is
$$
\int_{0}^{\infty} M_{\alpha, 1}(\theta) \theta^{|\beta_{2}|} d\theta = \frac{\Gamma(1 + |\beta_{2}|)}{\Gamma(1 + \alpha|\beta_{2}|)}.
$$
Therefore, for $0 < M < 1$, \eqref{multiplier condition} holds.

On the other hand, we verify that \eqref{multiplier condition} holds for $M \geq 1$.

First, we verify that for any $\varrho > 0$ and $\alpha \in (\frac{1}{2}, 1)$,
\begin{align*}
\left(a + \phi(|\xi|^2)\right)^{\frac{\varrho}{2}} \exp\left(\cos\left(\frac{\pi}{2\alpha}\right) \phi(|\xi|^2)^{\frac{1}{\alpha}} - i \sin\left(\frac{\pi}{2\alpha}\right) \phi(|\xi|^2)^{\frac{1}{\alpha}}\right) \in \mathcal{U}_p.
\end{align*}
Note that at this point, $\cos\left(\frac{\pi}{2\alpha}\right) < 0$, so this can be guaranteed by
\begin{align*}
&\left|D_{\xi}^{\beta}\left[\left(a + \phi(|\xi|^2)\right)^{\frac{\varrho}{2}} \exp\left(\cos\left(\frac{\pi}{2\alpha}\right) \phi(|\xi|^2)^{\frac{1}{\alpha}} - i \sin\left(\frac{\pi}{2\alpha}\right) \phi(|\xi|^2)^{\frac{1}{\alpha}}\right)\right]\right| \\
&\lesssim \left|\sum_{\beta_1 + \beta_2 = \beta} D^{\beta_1}_{\xi}\left(a + \phi(|\xi|^2)\right)^{\frac{\varrho}{2}} D^{\beta_2}_{\xi} \exp\left(\cos\left(\frac{\pi}{2\alpha}\right) \phi(|\xi|^2)^{\frac{1}{\alpha}} - i \sin\left(\frac{\pi}{2\alpha}\right) \phi(|\xi|^2)^{\frac{1}{\alpha}}\right)\right| \\
&\lesssim \sum_{\substack{\beta_1 + \beta_2 = \beta, \\ \frac{|\beta_2|}{2} \leq l \leq |\beta_2|}} |\xi|^{-|\beta|} \left|\left(a + \phi(|\xi|^2)\right)^{\frac{\varrho}{2} + \frac{l}{\alpha}} \exp\left(\cos\left(\frac{\pi}{2\alpha}\right) \phi(|\xi|^2)^{\frac{l}{\alpha}} - i \sin\left(\frac{\pi}{2\alpha}\right) \phi(|\xi|^2)^{\frac{1}{\alpha}}\right)\right| \\
&\lesssim |\xi|^{-|\beta|}.
\end{align*}

Next, we verify that for any $\alpha \in (0,\frac{1}{2})$ and $\alpha\in (\frac{1}{2},1)$, $M \geq 1$, $\sigma \in [0, 2]$,
\begin{align}\label{integral term}
\left(a + \phi(|\xi|^2)\right)^{\frac{\sigma}{2}} \int_{0}^{\infty} \frac{e^{-r^{\frac{1}{\alpha}}} \phi(|\xi|^2)}{r^2 + 2i \phi(|\xi|^2) r \cos(\alpha \pi) - (\phi(|\xi|^2))^2} \, dr
\end{align}
satisfies \eqref{multiplier condition}. Note that
\begin{align*}
&\left|D_{\xi}^{\beta} \int_{0}^{\infty} \frac{e^{-r^{\frac{1}{\alpha}}} \left(a + \phi(|\xi|^2)\right)^{\frac{\sigma}{2}} \phi(|\xi|^2)}{r^2 + 2i \phi(|\xi|^2) r \cos(\alpha \pi) - (\phi(|\xi|^2))^2} \, dr\right| \\
&\lesssim \sum_{\beta_1 + \beta_2 = \beta} \int_{0}^{\infty} \exp(-r^{\frac{1}{\alpha}}) \left|D_{\xi}^{\beta_1} \left(a + \phi(|\xi|^2)\right)^{\frac{\sigma}{2}} \phi(|\xi|^2)\right|\left|D_{\xi}^{\beta_2} \left(\frac{1}{r^2 + 2i \phi(|\xi|^2) r \cos(\alpha \pi) - (\phi(|\xi|^2))^2}\right)\right| \, dr \\
&\lesssim \sum_{\substack{\beta_1 + \beta_2 = \beta, \\ \frac{|\beta_2|}{2} \leq l \leq |\beta_2|}} |\xi|^{-|\beta|} \int_{0}^{\infty} \exp(-r^{\frac{1}{\alpha}}) \left(a + \phi(|\xi|^2)\right)^{\frac{\sigma}{2}} \phi(|\xi|^2)\left|\frac{r \cos(\alpha \pi) \phi(|\xi|^2)^l - \phi(|\xi|^2)^{2l}}{\phi(|\xi|^2)^{2(l + 1)}}\right| \, dr \\
&\leq C(\alpha, \beta, d, \delta) |\xi|^{-|\beta|},
\end{align*}
where we use the property of the Bernstein function, i.e., $\phi(x) \geq 1 \wedge x$, which shows that \eqref{integral term} satisfies \eqref{multiplier condition}.

For $\alpha=\frac{1}{2}$,
$$
E_{\frac{1}{2},1}(-i\phi(|\xi|^{2}))=\exp(-(\phi(|\xi|^{2}))^{2})-i\exp(-(\phi(|\xi|^{2}))^{2})\operatorname{erfi}(\phi(|\xi|^{2})).
$$
Clearly, $\left(a + \phi(|\xi|^2)\right)^{\frac{\sigma}{2}}\exp(-(\phi(|\xi|^{2}))^{2})\in\mathscr{U}_{p}(\mathbb{R}^{d})$; hence it remains to verify that for $M \geq 1$, $\sigma \in [0, 2]$,
$$
\left(a + \phi(|\xi|^2)\right)^{\frac{\sigma}{2}}\exp(-(\phi(|\xi|^{2}))^{2})\operatorname{erfi}(\phi(|\xi|^{2}))
$$
satisfies \eqref{multiplier condition}.  
First, note that when $M\geq 1$ and $|\xi|\sim M$, the asymptotic behaviour of the imaginary error function yields
$$
\operatorname{erfi}(\phi(|\xi|^{2}))\sim\frac{\exp((\phi(|\xi|^{2})^{2})}{\sqrt{\pi}\phi(|\xi|^{2})}\bigg(1+\frac{1}{2(\phi(|\xi|^{2}))^{2}}+O\!\left(\frac{1}{(\phi(|\xi|^{2}))^{3}}\right)\bigg),
$$
which implies that for $|\xi|\in[M/2,M]$,
$$
\left(a + \phi(|\xi|^2)\right)^{\frac{\sigma}{2}}\exp(-(\phi(|\xi|^{2}))^{2})\operatorname{erfi}(\phi(|\xi|^{2}))\sim\frac{\left(a + \phi(|\xi|^2)\right)^{\frac{\sigma}{2}}}{(\phi(|\xi|^{2}))}<\infty,
$$  
and 
\begin{align*}
	&\big|D^{\beta}_{\xi}\big[\left(a + \phi(|\xi|^2)\right)^{\frac{\sigma}{2}}\exp(-(\phi(|\xi|^{2}))^{2})\operatorname{erfi}(\phi(|\xi|^{2}))\big]\big|\\
	&\lesssim \bigg|D^{\beta}_{\xi}\big[\frac{\left(a + \phi(|\xi|^2)\right)^{\frac{\sigma}{2}}}{(\phi(|\xi|^{2}))}\big]\bigg|\\
	&\lesssim\sum_{\beta_1 + \beta_2 = \beta}\big|D^{\beta_{1}}_{\xi}\left(a + \phi(|\xi|^2)\right)^{\frac{\sigma}{2}}\big|\big|D_{\xi}^{\beta_{2}}\!\left(\frac{1}{\phi(|\xi|^{2})}\right)\big|\\
	&\lesssim|\xi|^{-|\beta|}.
\end{align*}
Therefore, from the above we conclude that \eqref{t=1,M.L.condition} holds.

Additionally, noting that $\phi(0^+) = 0$ and under Assumption H1 \eqref{S.C.condition}, for any $\xi \in \mathbb{R}^d$, consider the function $f(\lambda) = \phi(|\lambda \xi|^2)$. For every $t > 0$,
$$
\frac{f(1 + t^{\frac{\alpha}{2\delta}})}{t^{\alpha} \phi(|\xi|^2)} \geq \frac{(1 + t^{\frac{\alpha}{2\delta}})^{2\delta}}{t^{\alpha}} > 1,
$$
so by the Intermediate Value Theorem, we can take $\lambda^* \in (0, 1 + t^{\frac{\alpha}{2\delta}})$ such that $f(\lambda^*) = t^{\alpha} \phi(|\xi|^2)$. Combining \eqref{M.P.scailing} and \eqref{t=1,M.L.condition}, we obtain that for any $t > 0$,
\begin{align*}
\left\|\left(a + t^{\alpha} \phi(|\xi|^2)\right)^{\frac{\sigma}{2}} E_{\alpha, 1}(-i t^{\alpha} \phi(|\xi|^2))\right\|_{\mathcal{U}_p} &= \left\|\left(a + \phi(|\lambda^* \xi|^2)\right)^{\frac{\sigma}{2}} E_{\alpha, 1}(-i \phi(|\lambda^* \xi|^2))\right\|_{\mathcal{U}_p} \\
&= \left\|\left(a + \phi(|\xi|^2)\right)^{\frac{\sigma}{2}} E_{\alpha, 1}(-i \phi(|\xi|^2))\right\|_{\mathcal{U}_p},
\end{align*}
thus we conclude that
$$
\left(a + t^{\alpha} \phi(|\xi|^2)\right)^{\frac{\sigma}{2}} E_{\alpha, 1}(-i t^{\alpha} \phi(|\xi|^2)) \in \mathcal{U}_p(\mathbb{R}^d).
$$
\end{proof}
\begin{lemma}\label{M.L.Mittag-Leffer1}
For $1 < p < \infty$, any $a \geq 0$, and $\sigma \in [0, 4]$, we have that
$$
\left(a + t^{\alpha} \phi(|\xi|^{2})\right)^{\frac{\sigma}{2}} E_{\alpha, \alpha}(-it^{\alpha} \phi(|\xi|^{2})) \in \mathcal{U}_{p}(\mathbb{R}^{d}).
$$
\end{lemma}
\begin{proof}
	
	Analogously to Lemma~\ref{M.L.Mittag-Leffer}, it suffices to verify that for any $M > 0$,
	\begin{align}\label{multiplier condition1}
		\sum_{|\beta| \leq \lceil\frac{d}{2}\rceil + 1} \int_{\frac{M}{2} < |\xi| < 2M} \left|M^{|\beta|} D^{\beta}_{\xi} \left[\left(a + \phi(|\xi|^{2})\right)^{\frac{\sigma}{2}} E_{\alpha, \alpha}(-i\phi(|\xi|^{2}))\right]\right|^{2} d\xi \lesssim M^{d}.
	\end{align}
	
	For $0<M<1$, the subordination principle together with the arguments in Lemma~\ref{M.L.Mittag-Leffer} yields \eqref{multiplier condition1}.  
	
	For $M>1$, if $\alpha\in(1/2,1)$, applying Faà di Bruno's formula, similarly to Lemma~\ref{M.L.Mittag-Leffer}, we obtain that  
	$$
	\frac{1}{\alpha}\exp\!\big(-\frac{i\pi}{2}\frac{1-\alpha}{\alpha}\big)
	(\phi(|\xi|^{2}))^{\frac{1-\alpha}{\alpha}}
	\exp\!\big(e^{-i\frac{\pi}{2\alpha}}\phi(|\xi|^{2})^{\frac{1}{\alpha}}\big)
	$$
	satisfies \eqref{multiplier condition1}.  
	By the same reasoning, for $\alpha\in(0,1)$ with $\alpha\neq 1/2$ and $\sigma\in[0,4]$,  
	$$
	\left(a + \phi(|\xi|^2)\right)^{\frac{\sigma}{2}}\int_{0}^{\infty}
	\frac{r^{\frac{1}{\alpha}}e^{-r^{\frac{1}{\alpha}}}}{r^{2}
		+2i\phi(|\xi|^{2})r
		\cos(\alpha\pi)-(\phi(|\xi|^{2}))^{2}}\,dr
	$$  
	satisfies \eqref{multiplier condition1}.
	
	Furthermore, note that when $\alpha=1/2$,  
	$$
	E_{\frac{1}{2},\frac{1}{2}}(-i\phi(|\xi|^{2}))=\frac{1}{\sqrt{\pi}}-i\phi(|\xi|^{2})\exp(-(\phi(|\xi|^{2}))^{2})-\phi(|\xi|^{2})\exp(-(\phi(|\xi|^{2}))^{2})\operatorname{erfi}(\phi(|\xi|^{2})).
	$$  
	By Faà di Bruno's formula, it is clear that $\left(a + \phi(|\xi|^2)\right)^{\frac{\sigma}{2}}\phi(|\xi|^{2})\exp(-(\phi(|\xi|^{2}))^{2})\in\mathscr{U}_{p}(\mathbb{R}^{d})$; hence it remains to verify that for $M \geq 1$, $\sigma \in [0, 4]$,  
	$$
	\left(a + \phi(|\xi|^2)\right)^{\frac{\sigma}{2}}\bigg(\frac{1}{\sqrt{\pi}}-\phi(|\xi|^{2})\exp(-(\phi(|\xi|^{2}))^{2})\operatorname{erfi}(\phi(|\xi|^{2}))\bigg)
	$$  
	satisfies \eqref{multiplier condition1}.  
	
	Observe that when $M\geq 1$ and $|\xi|\sim M$, the asymptotic behaviour of the imaginary error function gives  
	$$
	\operatorname{erfi}(\phi(|\xi|^{2}))\sim\frac{\exp((\phi(|\xi|^{2})^{2})}{\sqrt{\pi}\phi(|\xi|^{2})}\bigg(1+\frac{1}{2(\phi(|\xi|^{2}))^{2}}+O\!\left(\frac{1}{(\phi(|\xi|^{2}))^{3}}\right)\bigg),
	$$  
	which implies that for $|\xi|\in[M/2,M]$,  
	$$
	\bigg|\left(a + \phi(|\xi|^2)\right)^{\frac{\sigma}{2}}\bigg(\frac{1}{\sqrt{\pi}}-\phi(|\xi|^{2})\exp(-(\phi(|\xi|^{2}))^{2})\operatorname{erfi}(\phi(|\xi|^{2}))\bigg)\bigg|\sim\frac{\left(a + \phi(|\xi|^2)\right)^{\frac{\sigma}{2}}}{(\phi(|\xi|^{2}))^{2}}<\infty.
	$$  
	\begin{align*}
		&\bigg|D^{\beta}_{\xi}\bigg[\left(a + \phi(|\xi|^2)\right)^{\frac{\sigma}{2}}\bigg(\frac{1}{\sqrt{\pi}}-\phi(|\xi|^{2})\exp(-(\phi(|\xi|^{2}))^{2})\operatorname{erfi}(\phi(|\xi|^{2}))\bigg)\bigg]\bigg|\\
		&\lesssim \bigg|D^{\beta}_{\xi}\big[\frac{\left(a + \phi(|\xi|^2)\right)^{\frac{\sigma}{2}}}{(\phi(|\xi|^{2}))^{2}}\big]\bigg|\\
		&\lesssim\sum_{\beta_1 + \beta_2 = \beta}\big|D^{\beta_{1}}_{\xi}\left(a + \phi(|\xi|^2)\right)^{\frac{\sigma}{2}}\big|\big|D_{\xi}^{\beta_{2}}\!\left(\frac{1}{(\phi(|\xi|^{2}))^{2}}\right)\big|\\
		&\lesssim|\xi|^{-|\beta|}.
	\end{align*}  
	In conclusion, \eqref{multiplier condition1} holds.
	
\end{proof}

Next, we establish some $L^{p}-L^{r}$ estimates and Sobolev estimates for the solution operators $\mathcal{S}_{\alpha,\phi}(t)$ and $\mathcal{P}_{\alpha,\phi}(t)$.

\begin{lemma}\label{operator S}
Let $\sigma \in [0, 2]$. The operator $\mathcal{S}_{\alpha,\phi}(t)$ satisfies the following properties:
\begin{align*}
\big\|\mathcal{S}_{\alpha,\phi}(t)g\big\|_{L^{r}} &\lesssim t^{-\frac{\alpha}{2\delta}\left(\frac{d}{p}-\frac{d}{r}\right)}\big\|g\big\|_{L^{p}}, \quad \text{where } 1 < p \leq r <
\begin{cases}
\frac{dp}{d - 2\delta}, & d \geq 2\delta, \\
\infty, & d < 2\delta,
\end{cases}
\end{align*}
and
\begin{align*}
\big\|\big(I - \phi(\Delta)\big)^{\frac{\sigma}{2}}\mathcal{S}_{\alpha,\phi}(t)g\big\|_{L^{r}} &\lesssim t^{-\frac{\alpha}{2\delta}\left(\delta\sigma + \frac{d}{p} - \frac{d}{r}\right)}\big\|g\big\|_{L^{p}}, \quad \text{where } 1 < p \leq r <
\begin{cases}
\frac{dp}{d - (2 - \sigma)\delta}, & d \geq (2 - \sigma)\delta, \\
\infty, & d < (2 - \sigma)\delta.
\end{cases}
\end{align*}
\end{lemma}

\begin{proof}
From Remark \ref{some remark of bernstein function}, Lemma \ref{M.L.Mittag-Leffer}, and the Gagliardo-Nirenberg inequality in Lemma \ref{G.N.inequality}, we have
\begin{align*}
\big\|\mathcal{S}_{\alpha,\phi}(t)g\big\|_{L^{r}} &\lesssim
\big\|\big(I - \phi(\Delta)\big)\mathcal{S}_{\alpha,\phi}(t)g\big\|_{L^{p}}^{\theta}
\big\|\mathcal{S}_{\alpha,\phi}(t)g\big\|_{L^{p}}^{1 - \theta} \\
&\lesssim t^{-\alpha\theta}\big\|\big(t^{\alpha} -   t^{\alpha}\phi(\Delta)\big)\mathcal{S}_{\alpha,\phi}(t)g\big\|_{L^{p}}^{\theta}
\big\|\mathcal{S}_{\alpha,\phi}(t)g\big\|_{L^{p}}^{1 - \theta} \\
&\lesssim t^{-\frac{\alpha}{2\delta}\left(\frac{d}{p} - \frac{d}{r}\right)}\big\|g\big\|_{L^{p}},
\end{align*}
where we use the condition
$$
\frac{d}{r} = \theta\left(\frac{d}{p} - 2\delta\right) + (1 - \theta)\frac{d}{p}.
$$
Similarly, we also have
\begin{align*}
&\big\|\big(I - \phi(\Delta)\big)^{\frac{\sigma}{2}}\mathcal{S}_{\alpha,\phi}(t)g\big\|_{L^{r}} \\
&\lesssim
\big\|\big(I - \phi(\Delta)\big)\mathcal{S}_{\alpha,\phi}(t)g\big\|_{L^{p}}^{\theta}
\big\|\big(I - \phi(\Delta)\big)^{\frac{\sigma}{2}}\mathcal{S}_{\alpha,\phi}(t)g\big\|_{L^{p}}^{1 - \theta} \\
&\lesssim t^{-\alpha\theta - \frac{\alpha\sigma}{2}(1 - \theta)}\big\|\big(t^{\alpha} - t^{\alpha}\phi(\Delta)\big)\mathcal{S}_{\alpha,\phi}(t)g\big\|_{L^{p}}^{\theta}
\big\|\big(t^{\alpha} - t^{\alpha}\phi(\Delta)\big)^{\frac{\sigma}{2}}\mathcal{S}_{\alpha,\phi}(t)g\big\|_{L^{p}}^{1 - \theta} \\
&\lesssim t^{-\frac{\alpha}{2\delta}\left(\delta\sigma + \frac{d}{p} - \frac{d}{r}\right)}\big\|g\big\|_{L^{p}},
\end{align*}
where we use the condition
$$
\frac{d}{r} - \delta\sigma = \theta\left(\frac{d}{p} - 2\delta\right) + (1 - \theta)\left(\frac{d}{p} - \sigma\delta\right).
$$
\end{proof}
\begin{lemma}\label{operator P}
Let $\sigma \in [0, 4]$. The operator $t^{\alpha-1}\mathcal{P}_{\alpha,\phi}(t)$ satisfies the following properties:
\begin{align*}
\big\|t^{\alpha-1}\mathcal{P}_{\alpha,\phi}(t)g\big\|_{L^{r}} &\lesssim t^{\alpha-1-\frac{\alpha}{2\delta}\left(\frac{d}{p}-\frac{d}{r}\right)}\big\|g\big\|_{L^{p}}, \quad \text{where } 1 < p \leq r <
\begin{cases}
\frac{dp}{d - 2\delta}, & d \geq 2\delta, \\
\infty, & d < 2\delta,
\end{cases}
\end{align*}
and
\begin{align*}
\big\|\big(I + \phi(\Delta)\big)^{\frac{\sigma}{2}}t^{\alpha-1}\mathcal{P}_{\alpha,\phi}(t)g\big\|_{L^{r}} &\lesssim t^{\alpha-1-\frac{\alpha}{2\delta}\left(\delta\sigma + \frac{d}{p} - \frac{d}{r}\right)}\big\|g\big\|_{L^{p}},
\end{align*}
where
$
1 < p \leq r <
\begin{cases}
\frac{dp}{d - (2 - \sigma)\delta}, & d \geq (2 - \sigma)\delta, \\
\infty, & d < (2 - \sigma)\delta.
\end{cases}
$
\end{lemma}
\begin{proof}
	Combining Remark \ref{some remark of bernstein function}, Lemma \ref{M.L.Mittag-Leffer1}, and the Gagliardo-Nirenberg inequality in Lemma \ref{G.N.inequality}, the proof follows by repeating the argument of Lemma \ref{operator S}.
\end{proof}

\begin{lemma}\label{Sobolev estimates operator S}
Let $a,b\in\mathbb{R}$, $b\geq a$, $1\leq \kappa\leq\infty$, and $1<p\leq r<\infty$ satisfy
$$
\frac{1}{2\delta}\left((b-a)\delta+\frac{d}{p}-\frac{d}{r}\right)<1.
$$
Then the following statements hold:
\begin{enumerate}[\rm(i)]
  \item
  $$
  \big\|\mathcal{S}_{\alpha,\phi}(t)g\big\|_{B^{b,\phi}_{r,\kappa}(\mathbb{R}^{d})}\lesssim t^{-\frac{\alpha}{2\delta}\left((b-a)\delta+\frac{d}{p}-\frac{d}{r}\right)}
  \big\|g\big\|_{B^{a,\phi}_{p,\kappa}(\mathbb{R}^{d})},
  $$
  \item
  If $b>a$, then we also have
  $$
  \big\|\mathcal{S}_{\alpha,\phi}(t)g\big\|_{B^{b,\phi}_{r,1}(\mathbb{R}^{d})}\lesssim t^{-\frac{\alpha}{2\delta}\left((b-a)\delta+\frac{d}{p}-\frac{d}{r}\right)}
  \big\|g\big\|_{B^{a,\phi}_{p,\infty}(\mathbb{R}^{d})}
  $$
\end{enumerate}
\end{lemma}

\begin{proof}
By Lemma \ref{M.L.Mittag-Leffer}, we have
\begin{align*}
\big\|\mathcal{S}_{\alpha,\phi}(t)g\big\|_{B^{b,\phi}_{r,\kappa}(\mathbb{R}^{d})}&=
\big\|(I-\phi(\Delta))^{\frac{b}{2}}\mathcal{S}_{\alpha,\phi}(t)g\big\|_{B^{0,\phi}_{r,\kappa}(\mathbb{R}^{d})}\\
&\lesssim\big\|(I-\phi(\Delta))^{\frac{b}{2}}\mathcal{S}_{\alpha,\phi}(t)g\big\|_{B^{c,\phi}_{p,\kappa}(\mathbb{R}^{d})}\quad \text{where } c=\frac{d}{\delta}\left(\frac{1}{p}-\frac{1}{r}\right)\\
&\lesssim\big\|(I-\phi(\Delta))^{\frac{b+c-a}{2}}\mathcal{S}_{\alpha,\phi}(t)g\big\|_{B^{a,\phi}_{p,\kappa}(\mathbb{R}^{d})}\\
&\lesssim t^{-\frac{\alpha}{2}(b+c-a)}\big\|(t^{\alpha}-t^{\alpha}\phi(\Delta))^{\frac{b+c-a}{2}}
\mathcal{S}_{\alpha,\phi}(t)g\big\|_{B^{a,\phi}_{p,\kappa}(\mathbb{R}^{d})}\\
&\lesssim t^{-\frac{\alpha}{2\delta}\left((b-a)\delta+\frac{d}{p}-\frac{d}{r}\right)}\big\|g\big\|_{B^{a,\phi}_{p,\kappa}(\mathbb{R}^{d})}.
\end{align*}
This completes the proof of (i).

Moreover, if $b>a$, we can choose $\varepsilon>0$ such that $b-\varepsilon\geq a$ and
$$
\frac{1}{2\delta}\left((b-a+\varepsilon)\delta+\frac{d}{p}-\frac{d}{r}\right)<1.
$$
From (i), we then have
\begin{align*}
\big\|\mathcal{S}_{\alpha,\phi}(t)g\big\|_{B^{b+\varepsilon,\phi}_{r,\infty}(\mathbb{R}^{d})}&\lesssim
t^{-\frac{\alpha}{2\delta}\left((b-a+\varepsilon)\delta+\frac{d}{p}-\frac{d}{r}\right)}
\big\|g\big\|_{B^{a,\phi}_{p,\infty}(\mathbb{R}^{d})},\\
\big\|\mathcal{S}_{\alpha,\phi}(t)g\big\|_{B^{b-\varepsilon,\phi}_{r,\infty}(\mathbb{R}^{d})}&\lesssim
t^{-\frac{\alpha}{2\delta}\left((b-a-\varepsilon)\delta+\frac{d}{p}-\frac{d}{r}\right)}
\big\|g\big\|_{B^{a,\phi}_{p,\infty}(\mathbb{R}^{d})}.
\end{align*}
By applying complex interpolation,
$$
\bigg(B^{b+\varepsilon,\phi}_{r,\infty}(\mathbb{R}^{d}),B^{b-\varepsilon,\phi}_{r,\infty}(\mathbb{R}^{d})\bigg)_{\theta,1}=
B^{b,\phi}_{r,1}(\mathbb{R}^{d}),
$$
we obtain
$$
\big\|\mathcal{S}_{\alpha,\phi}(t)g\big\|_{B^{b,\phi}_{r,1}(\mathbb{R}^{d})}\lesssim t^{-\frac{\alpha}{2\delta}\left((b-a)\delta+\frac{d}{p}-\frac{d}{r}\right)}\big\|g\big\|_{B^{a,\phi}_{p,\infty}(\mathbb{R}^{d})}
$$
This completes the proof.
\end{proof}
\begin{lemma}\label{Sobolev estimates operator P}
Let $a,b\in\mathbb{R}$, $b\geq a$, $1\leq \kappa\leq\infty$, and $1<p\leq r<\infty$ satisfy
$$
\frac{1}{2\delta}\left((b-a)\delta+\frac{d}{p}-\frac{d}{r}\right)<2.
$$
Then the following statements hold:
\begin{enumerate}[\rm(i)]
  \item
  $$
  \big\|t^{\alpha-1}\mathcal{P}_{\alpha,\phi}(t)g\big\|_{B^{b,\phi}_{r,\kappa}(\mathbb{R}^{d})}\lesssim t^{\alpha-1-\frac{\alpha}{2\delta}\left((b-a)\delta+\frac{d}{p}-\frac{d}{r}\right)}
  \big\|g\big\|_{B^{a,\phi}_{p,\kappa}(\mathbb{R}^{d})},
  $$
  \item
  If $b>a$, then we also have
  $$
  \big\|t^{\alpha-1}\mathcal{P}_{\alpha,\phi}(t)g\big\|_{B^{b,\phi}_{r,1}(\mathbb{R}^{d})}\lesssim t^{\alpha-1-\frac{\alpha}{2\delta}\left((b-a)\delta+\frac{d}{p}-\frac{d}{r}\right)}
  \big\|g\big\|_{B^{a,\phi}_{p,\infty}(\mathbb{R}^{d})}
  $$
\end{enumerate}
\end{lemma}
\begin{proof}
The proof follows by repeating the argument of Lemma \ref{Sobolev estimates operator S} and by using the Lemma \ref{operator P}.
\end{proof}

\begin{remark}\label{homonuous case}
In Lemma {\rm\ref{Sobolev estimates operator S}}, Lemma {\rm\ref{Sobolev estimates operator P}}, if we replace $B^{b,\phi}_{p,\kappa}(\mathbb{R}^{d})$, $B^{a,\phi}_{p,\kappa}(\mathbb{R}^{d})$, $B^{b,\phi}_{p,1}(\mathbb{R}^{d})$, and $B^{a,\phi}_{p,\infty}(\mathbb{R}^{d})$ with $\dot{B}^{b,\phi}_{p,\kappa}(\mathbb{R}^{d})$, $\dot{B}^{a,\phi}_{p,\kappa}(\mathbb{R}^{d})$, $\dot{B}^{b,\phi}_{p,1}(\mathbb{R}^{d})$, and $\dot{B}^{a,\phi}_{p,\infty}(\mathbb{R}^{d})$, respectively, the results remain valid.
\end{remark}
\section{Main Results}

In this section, we construct the Main Theorem in this paper. First, we give the definition of the triple admissible family $(p, r, q)$.

\begin{definition}
The triple admissible family $(p, r, q)$ is defined as
\begin{align*}
\frac{1}{q} = \frac{1}{2\delta} \left( \frac{d}{p} - \frac{d}{r} \right), \text{ where }
1 < p \leq r <
\begin{cases}
\frac{dp}{d - 2\delta}, & d \geq 2\delta, \\
\infty, & d < 2\delta.
\end{cases}
\end{align*}
\end{definition}

For the nonlinear term $g(w)$, we give the following assumption, which follows from Su \cite{Su}.

\textbf{Assumption H2}
\quad $g \in C(\mathbb{C}, \mathbb{C})$, $g(0) = 0$, and there exists a constant $\kappa > 0$ such that
$$
\big|g(u) - g(w)\big| \lesssim \big(|u|^{\kappa} + |w|^{\kappa}\big) |u - w|,
$$
which implies that for any $1 \leq r \leq \infty$, we have
$$
\big\|g(u) - g(w)\big\|_{L^{\frac{r}{1+\kappa}}(\mathbb{R}^{d})} \lesssim \big(\big\|u\big\|^{\kappa}_{L^{r}(\mathbb{R}^{d})} + \big\|w\big\|_{L^{r}(\mathbb{R}^{d})}^{\kappa}\big) \big\|u - w\big\|_{L^{r}(\mathbb{R}^{d})}.
$$

Furthermore, let $(p, r, q)$ be the triple admissible family. We consider the following Banach space $\mathcal{X}^{\alpha}_{T}$ (we also denote it as $\mathcal{X}^{\alpha}$ when $T = \infty$):
$$
\mathcal{X}^{\alpha}_{T} = \big\{w : w \in C\big([0, T]; \dot{B}^{\gamma_{0}, \phi}_{p_{0}, \infty}\big) \cap C_{\alpha, q}\big((0, T]; L^{r}(\mathbb{R}^{d})\big)\big\}, \text{ where } \gamma_{0} = \frac{1}{\delta}\left(\frac{d}{p_{0}} - \frac{d}{p}\right), 1 < p \leq p_{0} \leq r < \infty,
$$
with the norm
$$
\big\|w\big\|_{\mathcal{X}^{\alpha}_{T}} = \sup_{t \in [0, T]}\big\|w(t)\big\|_{\dot{B}^{\gamma_{0}, \phi}_{p_{0}, \infty}} + \sup_{t \in [0, T]} t^{\frac{\alpha}{q}} \big\|w(t)\big\|_{L^{r}(\mathbb{R}^{d})}.
$$
\begin{theorem}\label{Local exist of mild solution}
If the \textbf{Assumption H2} is satisfied and let $(p, r, q)$ be an admissible triple family, satisfying
\begin{align*}
\begin{cases}
(1+\kappa) \vee p < r < p\big(1+\kappa\big) \\
1 \vee \frac{d\kappa}{2\delta} < p < p_{0} \leq r \\
\frac{d}{p} - \frac{2\delta}{1+\kappa} < \frac{d}{r} < \frac{2\delta}{\kappa}
\end{cases}
\end{align*}
Then for initial value $w_{0} \in \dot{B}^{\gamma_{0}, \phi}_{p_{0}, \infty}$, the E.q. \eqref{E.Q.T.F.S.E} possesses a local mild solution on $\mathcal{X}_{T^{*}}^{\alpha}$ for some $T^{*}>0$. Moreover, for any \( T' \in (0, T^{\ast}) \), there exists a neighborhood \( V \) of \( w_{0} \) in the space \( \dot{B}^{\gamma_{0}, \phi}_{p_{0}, \infty} \) such that the mapping \( \tilde{w}_{0} \mapsto \tilde{w} \) from \( V \) to \( \mathcal{X}_{T'}^{\alpha} \) is Lipschitz continuous.
\end{theorem}
\begin{proof}
By Definition \ref{D.f.mild solution}, we define the mapping
\begin{align}\label{operator formula}
\Theta w(t, x) = \mathcal{S}_{\alpha, \phi}(t)w_{0}(x) + \frac{1}{i}
\int_{0}^{t}(t - \tau)^{\alpha - 1}\mathcal{P}_{\alpha, \phi}(t - \tau)g(w(\tau, x)) \, d\tau, \quad t \in [0, T].
\end{align}
We aim to prove that there exists $T^{*} > 0$ such that the mapping $\Theta$ has a fixed point in $\mathcal{X}_{T^{*}}^{\alpha}$.

First, we verify that for any $T > 0$, the operator $\Theta w \in C\big([0, T]; \dot{B}^{\gamma_{0}, \phi}_{p_{0}, \infty}\big) \cap C_{\alpha, q}\big((0, T]; L^{r}(\mathbb{R}^{d})\big)$. For any $0 \leq t_{1} < t_{2} \leq T$, we obtain
\begin{align*}
\Theta w(t_{2}) - \Theta w(t_{1}) &= \mathcal{S}_{\alpha, \phi}(t_{2})w_{0} - \mathcal{S}_{\alpha, \phi}(t_{1})w_{0} \\
&\quad + \frac{1}{i}\int_{t_{1}}^{t_{2}}(t_{2} - \tau)^{\alpha - 1}\mathcal{P}_{\alpha, \phi}(t_{2} - \tau)g(w(\tau)) \, d\tau \\
&\quad + \frac{1}{i}\int_{0}^{t_{1}}(t_{2} - \tau)^{\alpha - 1}\mathcal{P}_{\alpha, \phi}(t_{2} - \tau)g(w(\tau)) \\
&\quad - (t_{1} - \tau)^{\alpha - 1}\mathcal{P}_{\alpha, \phi}(t_{1} - \tau)g(w(\tau)) \, d\tau.
\end{align*}
By Lemma \ref{Sobolev estimates operator S} and Remark \ref{homonuous case}, for $w_{0} \in \dot{B}^{\gamma_{0}, \phi}_{p_{0}, \infty}$, it is clear that $\mathcal{S}_{\alpha, \phi}(t)w_{0} \in \dot{B}^{\gamma_{0}, \phi}_{p_{0}, \infty}$. Hence, we have
$$
\big\|\mathcal{S}_{\alpha, \phi}(t_{2})w_{0} - \mathcal{S}_{\alpha, \phi}(t_{1})w_{0}\big\|_{\dot{B}^{\gamma_{0}, \phi}_{p_{0}, \infty}} \rightarrow 0, \quad \text{as } t_{2} \rightarrow t_{1}.
$$
by using the Lebesgue dominated convergence theorem.

By Lemma \ref{embedding theorem}, we have
$$
L^{p}(\mathbb{R}^{d}) \hookrightarrow \dot{B}^{0, \phi}_{p, \infty}(\mathbb{R}^{d}) \hookrightarrow \dot{B}^{\gamma_{0}, \phi}_{p_{0}, \infty}(\mathbb{R}^{d}),
$$
and by Lemma \ref{Sobolev estimates operator P} and Remark \ref{homonuous case}, we have
\begin{align*}
&\bigg\|\int_{t_{1}}^{t_{2}}(t_{2} - \tau)^{\alpha - 1}\mathcal{P}_{\alpha, \phi}(t_{2} - \tau)g(w(\tau)) \, d\tau\bigg\|_{\dot{B}^{\gamma_{0}, \phi}_{p_{0}, \infty}} \\
&\lesssim \int_{t_{1}}^{t_{2}}(t_{2} - \tau)^{\alpha - 1 - \frac{\alpha d}{2\delta}\left(\frac{1 + \kappa}{r} - \frac{1}{p}\right)}\big\|w(\tau)\big\|_{L^{r}}^{\kappa + 1} \, d\tau \\
&\lesssim \int_{t_{1}}^{t_{2}}(t_{2} - \tau)^{\alpha - 1 - \frac{\alpha \kappa d}{2\delta p}}\tau^{-\frac{\alpha(\kappa + 1)}{q}}\left(\tau^{\frac{\alpha}{q}}\big\|w(\tau)\big\|_{L^{r}}\right)^{\kappa + 1} \, d\tau \\
&\lesssim (t_{2} - t_{1})^{\alpha - \frac{\alpha \kappa d}{2\delta p}}\big\|w\big\|_{\mathcal{X}^{\alpha}_{T}} \rightarrow 0, \quad \text{as } t_{2} \rightarrow t_{1}.
\end{align*}
Hence, by the Lebesgue dominated convergence theorem,
$$
\bigg\|\int_{t_{1}}^{t_{2}}(t_{2} - \tau)^{\alpha - 1}\mathcal{P}_{\alpha, \phi}(t_{2} - \tau)g(w(\tau)) \, d\tau\bigg\|_{\dot{B}^{\gamma_{0}, \phi}_{p_{0}, \infty}} \rightarrow 0, \quad \text{as } t_{2} \rightarrow t_{1}.
$$
Similarly, we note that
\begin{align*}
&\bigg\|\int_{0}^{t_{1}}(t_{2} - \tau)^{\alpha - 1}\mathcal{P}_{\alpha, \phi}(t_{2} - \tau)g(w(\tau))- (t_{1} - \tau)^{\alpha - 1}\mathcal{P}_{\alpha, \phi}(t_{1} - \tau)g(w(\tau)) \, d\tau\bigg\|_{\dot{B}^{\gamma_{0}, \phi}_{p_{0}, \infty}} \\
&\lesssim \int_{0}^{t_{1}}(t_{1} - \tau)^{\alpha - 1 - \frac{\alpha d}{2\delta}\left(\frac{1 + \kappa}{r} - \frac{1}{p}\right)}\tau^{-\frac{\alpha(\kappa + 1)}{q}}\left(\tau^{\frac{\alpha}{q}}\big\|w(\tau)\big\|_{L^{r}}\right)^{\kappa + 1} \, d\tau \\
&\lesssim B\left(\alpha - \frac{\alpha d}{2\delta}\left(\frac{1 + \kappa}{r} - \frac{1}{p}\right), 1 - \frac{\alpha(\kappa + 1)}{q}\right)\big\|w\big\|_{\mathcal{X}^{\alpha}_{T}},
\end{align*}
and we also get
$$
\bigg\|\int_{0}^{t_{1}}(t_{2} - \tau)^{\alpha - 1}\mathcal{P}_{\alpha, \phi}(t_{2} - \tau)g(w(\tau)) \\
- (t_{1} - \tau)^{\alpha - 1}\mathcal{P}_{\alpha, \phi}(t_{1} - \tau)g(w(\tau)) \, d\tau\bigg\|_{\dot{B}^{\gamma_{0}, \phi}_{p_{0}, \infty}} \rightarrow 0, \quad \text{as } t_{2} \rightarrow t_{1}.
$$
by the Lebesgue dominated convergence theorem.
Therefore, we obtain that $\Theta w \in C\big([0, T]; \dot{B}^{\gamma_{0}, \phi}_{p_{0}, \infty}\big)$.

Next, we verify that $\Theta w \in C_{\alpha, q}\big((0, T]; L^{r}(\mathbb{R}^{d})\big)$. For any $0 < t_{1} \leq t_{2} \leq T$, we have that
\begin{align*}
&\big\|t_{2}^{\frac{\alpha}{q}}\mathcal{S}_{\alpha,\phi}(t_{2})w_{0}
-t_{1}^{\frac{\alpha}{q}}\mathcal{S}_{\alpha,\phi}(t_{1})w_{0}\big\|_{L^{r}(\mathbb{R}^{d})}\\
&\lesssim (t_{2}^{\frac{\alpha}{q}}-t_{1}^{\frac{\alpha}{q}})\big\|\mathcal{S}_{\alpha,\phi}(t_{2})w_{0}\big\|_{L^{r}(\mathbb{R}^{d})}+
t_{1}^{\frac{\alpha}{q}}
\big\|\mathcal{S}_{\alpha,\phi}(t_{2})w_{0}-\mathcal{S}_{\alpha,\phi}(t_{1})w_{0}\big\|_{L^{r}(\mathbb{R}^{d})}.
\end{align*}
By Lemma \ref{Sobolev estimates operator S}, and noting that $\dot{B}^{0,\phi}_{r,1}(\mathbb{R}^{d})\hookrightarrow L^{r}(\mathbb{R}^{d})$, we obtain that
\begin{align*}
(t_{2}^{\frac{\alpha}{q}}-t_{1}^{\frac{\alpha}{q}})\big\|\mathcal{S}_{\alpha,\phi}(t_{2})w_{0}\big\|_{L^{r}(\mathbb{R}^{d})}
&\lesssim(t_{2}^{\frac{\alpha}{q}}-t_{1}^{\frac{\alpha}{q}})
\big\|\mathcal{S}_{\alpha,\phi}(t_{2})w_{0}\big\|_{\dot{B}^{0,\phi}_{r,1}(\mathbb{R}^{d})}\\
&\lesssim \big(1-(\frac{t_{1}}{t_{2}})^{\frac{\alpha}{q}}\big)
\big\|w_{0}\big\|_{\dot{B}^{\gamma_{0},\phi}_{p_{0},\infty}(\mathbb{R}^{d})}\\
&\rightarrow 0,\text{ as }t_{2}\rightarrow t_{1}.
\end{align*}
On the other hand, note that
\begin{align*}
\mathcal{S}_{\alpha,\phi}(t_{2})w_{0}-\mathcal{S}_{\alpha,\phi}(t_{1})w_{0}&=
\int_{t_{1}}^{t_{2}}\frac{d}{dt}\mathcal{S}_{\alpha,\phi}(t)w_{0}dt\\
&=-i\int_{t_{1}}^{t_{2}}t^{\alpha-1}\phi(\Delta)\mathcal{P}_{\alpha,\phi}(t)w_{0}dt,
\end{align*}
By Remark \ref{homonuous case}, Lemma \ref{Sobolev estimates operator P} and noting that
$$
\frac{1}{2\delta}\big((2-\gamma_{0})\delta+\frac{d}{p}-\frac{d}{r}\big)
=1+\frac{d}{2\delta}\big(\frac{1}{p_{0}}-\frac{1}{r}\big)<2,
$$
we get that
\begin{align*}
&t_{1}^{\frac{\alpha}{q}}
\big\|\mathcal{S}_{\alpha,\phi}(t_{2})w_{0}-\mathcal{S}_{\alpha,\phi}(t_{1})w_{0}\big\|_{L^{r}(\mathbb{R}^{d})}\\
&\lesssim t_{1}^{\frac{\alpha}{q}}\int_{t_{1}}^{t_{2}}t^{\alpha-1}\big\|\phi(\Delta)
\mathcal{P}_{\alpha,\phi}(t)w_{0}\big\|_{\dot{B}^{0,\phi}_{r,1}(\mathbb{R}^{d})}\,dt\\
&\lesssim t_{1}^{\frac{\alpha}{q}}\int_{t_{1}}^{t_{2}}\big\|t^{\alpha-1}
\mathcal{P}_{\alpha,\phi}(t)w_{0}\big\|_{\dot{B}^{2,\phi}_{r,1}(\mathbb{R}^{d})}\,dt\\
&\lesssim t_{1}^{\frac{\alpha}{q}}\int_{t_{1}}^{t_{2}}t^{\alpha-1-\alpha-\frac{\alpha}{2\delta}(\frac{d}{p_{0}}-\frac{d}{r})}
\big\|w_{0}\big\|_{\dot{B}^{\gamma_{0},\phi}_{p_{0},\infty}(\mathbb{R}^{d})}dt\\
&\lesssim \big(t_{1}^{\frac{\alpha}{2\delta}(\frac{d}{p}-\frac{d}{p_{0}})}
-\big(\frac{t_{1}}{t_{2}}\big)^{\frac{\alpha}{q}}t_{2}^{\frac{\alpha}{2\delta}(\frac{d}{p}-\frac{d}{p_{0}})}\big)
\big\|w_{0}\big\|_{\dot{B}^{\gamma_{0},\phi}_{p_{0},\infty}(\mathbb{R}^{d})}\rightarrow
0, \text{ as }t_{2}\rightarrow t_{1},
\end{align*}
Moreover, noting that
\begin{align*}
&t^{\frac{\alpha}{q}}_{2}\bigg\|\int_{t_{1}}^{t_{2}}(t_{2}-\tau)^{\alpha-1}
\mathcal{P}_{\alpha,\phi}(t_{2}-\tau)g(w(\tau))\,d\tau\bigg\|
_{L^{r}(\mathbb{R}^{d})}\\
&\lesssim t_{2}^{\frac{\alpha}{q}}\int_{t_{1}}^{t_{2}}(t_{2}-\tau)^{\alpha-1-\frac{\alpha d}{2\delta}(\frac{\kappa+1}{r}-\frac{1}{r})}\big\|w(\tau)\big\|_{L^{r}}^{\kappa+1}\,d\tau\\
&\lesssim t_{2}^{\frac{\alpha}{q}}\int_{t_{1}}^{t_{2}}(t_{2}-\tau)^{\alpha-1-\frac{\alpha d}{2\delta}(\frac{\kappa+1}{r}-\frac{1}{r})}
\tau^{-\frac{\alpha (\kappa+1)}{q}}\big(\tau^{\frac{\alpha}{q}}\big\|w(\tau)\big\|_{L^{r}}\big)^{\kappa+1}\,d\tau\\
&\lesssim t_{2}^{\alpha-\frac{\alpha\kappa d}{2\delta p}}
\int_{\frac{t_{1}}{t_{2}}}^{1}\big(1-\tau\big)^{\alpha-1-\frac{\alpha \kappa d}{2\delta r}}
\tau^{-\frac{\alpha (\kappa+1)}{q}}\,d\tau\big\|w\big\|_{\mathcal{X}^{\alpha}_{T}}\rightarrow 0,\text{ as }
t_{2}\rightarrow t_{1},
\end{align*}
and similarly, we have that
\begin{align*}
&\big(t_{2}^{\frac{\alpha}{q}}-t_{1}^{\frac{\alpha}{q}}\big)
\bigg\|\int_{0}^{t_{1}}(t_{2}-\tau)^{\alpha-1}\mathcal{P}_{\alpha,\phi}(t_{2}-\tau)g(w(\tau))
-(t_{1}-\tau)^{\alpha-1}\mathcal{P}_{\alpha,\phi}(t_{1}-\tau)g(w(\tau))\,d\tau\bigg\|_{L^{r}(\mathbb{R}^{d})}\\
&\lesssim\big(t_{2}^{\frac{\alpha}{q}}-t_{1}^{\frac{\alpha}{q}}\big)\int_{0}^{t_{1}}(t_{1}-\tau)^{\alpha-1-\frac{\alpha d}{2\delta}(\frac{\kappa+1}{r}-\frac{1}{r})}\big\|w(\tau)\big\|_{L^{r}}^{\kappa+1}\,d\tau\\
&\lesssim\big(t_{2}^{\frac{\alpha}{q}}-t_{1}^{\frac{\alpha}{q}}\big)t_{1}^{\alpha-\frac{\alpha\kappa d}{2\delta p}-\frac{\alpha}{q}}B\big(\alpha-\frac{\alpha\kappa d}{2\delta r},1-\frac{\alpha(\kappa+1)}{q}\big)\big\|w\big\|_{\mathcal{X}^{\alpha}_{T}}\\
&\lesssim \big(\big(\frac{t_{2}}{t_{1}}\big)^{\frac{\alpha}{q}}-1\big)t_{1}^{\alpha-\frac{\alpha\kappa d}{2\delta p}}
\big\|w\big\|_{\mathcal{X}^{\alpha}_{T}}\rightarrow 0,\text{ as }
t_{2}\rightarrow t_{1}.
\end{align*}
Now we have verified that $\Theta w\in C\big([0,T];\dot{B}^{\gamma_{0},\phi}_{p_{0},\infty}\big)\cap C_{\alpha,q}\big((0,T];L^{r}(\mathbb{R}^{d})\big)$.

From the above process, we also easily get
\begin{align*}
\big\|\Theta w\big\|_{\mathcal{X}^{\alpha}_{T}} &\lesssim \sup_{t\in [0,T]}\big\|\Theta w(t)\big\|_{\dot{B}^{\gamma_{0},\phi}_{p_{0},\infty}} + \sup_{t\in [0,T]} t^{\frac{\alpha}{q}} \big\|\Theta w(t)\big\|_{L^{r}(\mathbb{R}^{d})},
\end{align*}
and
\begin{align*}
\big\|\Theta w(t)\big\|_{\dot{B}^{\gamma_{0},\phi}_{p_{0},\infty}} &\lesssim \big\|\mathcal{S}_{\alpha,\phi}(t)w_{0}\big\|_{\dot{B}^{\gamma_{0},\phi}_{p_{0},\infty}} + \bigg\|\frac{1}{i}\int_{0}^{t}(t-\tau)^{\alpha-1}\mathcal{P}_{\alpha,\phi}(t_{2}-\tau)g(w(\tau))\,d\tau
\bigg\|_{\dot{B}^{\gamma_{0},\phi}_{p_{0},\infty}} \\
&\lesssim \big\|w_{0}\big\|_{\dot{B}^{\gamma_{0},\phi}_{p_{0},\infty}} + \int_{0}^{t}(t-\tau)^{\alpha-1-\frac{\alpha d}{2\delta}\left(\frac{1+\kappa}{r}-\frac{1}{p}\right)} \tau^{-\frac{\alpha(\kappa+1)}{q}} \left(\tau^{\frac{\alpha}{q}} \big\|w(\tau)\big\|_{L^{r}}\right)^{\kappa+1}\,d\tau \\
&\lesssim \big\|w_{0}\big\|_{\dot{B}^{\gamma_{0},\phi}_{p_{0},\infty}} + T^{\alpha-\frac{\alpha\kappa d}{2\delta p}} \big\|w\big\|^{\kappa+1}_{\mathcal{X}^{\alpha}_{T}},
\end{align*}
\begin{align*}
t^{\frac{\alpha}{q}} \big\|\Theta w(t)\big\|_{L^{r}} &\lesssim t^{\frac{\alpha}{q}} \big\|\mathcal{S}_{\alpha,\phi}(t)w_{0}\big\|_{L^{r}} + t^{\frac{\alpha}{q}} \int_{0}^{t}(t-\tau)^{\alpha-1} \bigg\|\mathcal{P}_{\alpha,\phi}(t_{2}-\tau)g(w(\tau))\bigg\|_{L^{r}}\,d\tau \\
&\lesssim t^{\frac{\alpha}{q}} \big\|\mathcal{S}_{\alpha,\phi}(t)w_{0}\big\|_{\dot{B}^{0,\phi}_{r,1}} + t^{\frac{\alpha}{q}} \int_{0}^{t}(t-\tau)^{\alpha-1-\frac{\alpha d}{2\delta}\left(\frac{\kappa+1}{r}-\frac{1}{r}\right)} \tau^{-\frac{\alpha (\kappa+1)}{q}} \left(\tau^{\frac{\alpha}{q}} \big\|w(\tau)\big\|_{L^{r}}\right)^{\kappa+1}\,d\tau \\
&\lesssim \big\|w_{0}\big\|_{\dot{B}^{\gamma_{0},\phi}_{p_{0},\infty}} + T^{\alpha-\frac{\alpha\kappa d}{2\delta p}} \big\|w\big\|^{\kappa+1}_{\mathcal{X}^{\alpha}_{T}}.
\end{align*}
Since $w_{0} \in \dot{B}^{\gamma_{0},\phi}_{p_{0},\infty}$, there exists a constant $R > 0$ such that $\big\|w_{0}\big\|_{\dot{B}^{\gamma_{0},\phi}_{p_{0},\infty}} \leq \frac{R}{2C}$. We can choose $T^{*} > 0$ such that
$$
C(T^{*})^{\alpha-\frac{\alpha\kappa d}{2\delta p}} R^{\kappa} < \frac{1}{2}.
$$
Consider the closed ball $\mathcal{B}_{R}$ defined as
$$
\mathcal{B}_{R} = \big\{w \in \mathcal{X}_{T^{*}}^{\alpha} : \big\|w\big\|_{\mathcal{X}_{T^{*}}^{\alpha}} \leq R\big\}.
$$
For any $w \in \mathcal{B}_{R}$, we have
\begin{align*}
\big\|\Theta w\big\|_{\mathcal{X}^{\alpha}_{T^{*}}} &\lesssim \big\|w_{0}\big\|_{\dot{B}^{\gamma_{0},\phi}_{p_{0},\infty}} + (T^{*})^{\alpha-\frac{\alpha\kappa d}{2\delta p}} \big\|w\big\|^{\kappa+1}_{\mathcal{X}^{\alpha}_{T}} \\
&\leq C \frac{R}{2C} + \frac{R}{2} \leq R,
\end{align*}
which shows that the operator $\Theta$ maps $\mathcal{B}_{R}$ into itself.

Moreover, for any $u, w \in \mathcal{B}_{R}$, and noting \textbf{Assumption H2}, we have
\begin{align*}
\big\|\Theta u - \Theta w\big\|_{\mathcal{X}^{\alpha}_{T^{*}}} &\lesssim \sup_{t\in [0,T^{*}]}\bigg\|\int_{0}^{t}(t-\tau)^{\alpha-1}\mathcal{P}_{\alpha,\phi}(t_{2}-\tau)(g(u(\tau)) - g(w(\tau)))\,d\tau\bigg\|_{\dot{B}^{\gamma_{0},\phi}_{p_{0},\infty}} \\
&\quad + \sup_{t\in [0,T^{*}]}\bigg\|\int_{0}^{t}(t-\tau)^{\alpha-1}\mathcal{P}_{\alpha,\phi}(t_{2}-\tau)(g(u(\tau)) - g(w(\tau)))\,d\tau\bigg\|_{L^{r}(\mathbb{R}^{d})} \\
&\lesssim (T^{*})^{\alpha-\frac{\alpha\kappa d}{2\delta p}} \left(\big\|u\big\|^{\kappa}_{\mathcal{X}^{\alpha}_{T^{*}}} + \big\|w\big\|^{\kappa}_{\mathcal{X}^{\alpha}_{T^{*}}}\right) \big\|u - w\big\|_{\mathcal{X}^{\alpha}_{T^{*}}} \\
&< \big\|u - w\big\|_{\mathcal{X}^{\alpha}_{T^{*}}}.
\end{align*}
Therefore, by the fixed point theorem, the operator $\Theta$ has a unique fixed point $w$, which is the unique mild solution of the E.q. \eqref{E.Q.T.F.S.E}.
\end{proof}
Next, we establish the global well-posedness of mild solutions with respect to E.q. \eqref{E.Q.T.F.S.E}.
\begin{theorem}\label{global well-posedbess mild solution}
If the \textbf{Assumption H2} is satisfied and let \((p, r, q)\) be an admissible triplet satisfying:
\begin{align*}
\begin{cases}
(1+\kappa) \vee p < r < p(1+\kappa), \\
1 < \frac{d\kappa}{2\delta} = p < p_{0} \leq r, \\
\frac{d}{p} - \frac{2\delta}{1+\kappa} < \frac{d}{r} < \frac{2\delta}{\kappa}.
\end{cases}
\end{align*}
Then, for \(w_{0} \in \dot{B}^{\gamma_{0},\phi}_{p_{0},\infty}\) with sufficiently small norm, E.q. \eqref{E.Q.T.F.S.E} admits a unique global mild solution \(w\) in \(\mathcal{X}^{\alpha}\). Moreover, there exists a neighborhood \(V\) of \(w_{0}\) in \(\dot{B}^{\gamma_{0}, \phi}_{p_{0}, \infty}\) such that the mapping \(\tilde{w}_{0} \mapsto \tilde{w}\) from \(V\) to \(\mathcal{X}^{\alpha}\) is Lipschitz continuous.
\end{theorem}

\begin{proof}
The proof follows a structure entirely analogous to Theorem \ref{Local exist of mild solution}, with appropriate modifications.

By Definition \ref{D.f.mild solution}, we need to verify that the operator \(\Theta\), defined as
\[
\Theta w(t, x) = \mathcal{S}_{\alpha, \phi}(t)w_{0}(x) + \frac{1}{i} \int_{0}^{t} (t - \tau)^{\alpha - 1} \mathcal{P}_{\alpha, \phi}(t - \tau)g(w(\tau, x)) \, d\tau, \quad t \in [0, \infty),
\]
has a unique fixed point.

Noting that \(w_{0} \in \dot{B}^{\gamma_{0},\phi}_{p_{0},\infty}\) with sufficiently small norm, there exists a small constant \(R > 0\) such that \(\|w_{0}\|_{\dot{B}^{\gamma_{0},\phi}_{p_{0},\infty}} < \frac{R}{2C}\) and \(CR^{\kappa} < \frac{1}{2}\).

Consider the closed ball \(\mathcal{B}_R\) defined by
\[
\mathcal{B}_R = \left\{ w \in \mathcal{X}^{\alpha} : \|w\|_{\mathcal{X}^{\alpha}} \leq  R \right\}.
\]
Since \(p = \frac{d\kappa}{2\delta}\), we have
\[
\frac{\kappa d}{2\delta r} + \frac{\kappa + 1}{q} = 1.
\]
Following a procedure entirely similar to Theorem \ref{Local exist of mild solution}, we deduce that \(\Theta w \in C([0, \infty); \dot{B}^{\gamma_{0}, \phi}_{p_{0}, \infty}) \cap C_{\alpha, q}((0, \infty); L^{r}(\mathbb{R}^{d}))\), and
\begin{align*}
\|\Theta w\|_{\mathcal{X}^{\alpha}}&\leq  \sup_{t\geq 0}\|\Theta w(t)\|_{\dot{B}^{\gamma_{0},\phi}_{p_{0},\infty}}+\sup_{t>0}t^{\frac{\alpha}{q}}\|\Theta w\|_{L^{r}(\mathbb{R}^{d})}\leq R
\end{align*}
Additionally,
\[
\|\Theta u - \Theta w\|_{\mathcal{X}^{\alpha}} < \|u - w\|_{\mathcal{X}^{\alpha}}.
\]
Therefore, by the Banach fixed-point theorem, the proof is complete.
\end{proof}
\begin{corollary}
In Theorem {\rm\ref{Local exist of mild solution}} and Theorem {\rm\ref{global well-posedbess mild solution}}, the restriction $r < p(1+\kappa)$ is unnecessary.
\end{corollary}

\begin{proof}
We only explain the case for Theorem \ref{Local exist of mild solution}, as the proof for Theorem \ref{global well-posedbess mild solution} is entirely analogous. Indeed,
if $r \geq p(1+\kappa)$, then we can choose a sufficiently small $\varsigma > 0$ such that $(p, \tilde{r}, \tilde{q})$ forms an admissible triple, where $\tilde{r} = p(1+\kappa) - \varsigma$, and
$$
\frac{d}{2\delta}\left(\frac{1+\kappa}{\tilde{r}} - \frac{1}{r}\right) < 1.
$$
By Theorem \ref{Local exist of mild solution}, the equation \eqref{E.Q.T.F.S.E} possesses a unique mild solution $w$ in
$C\big([0, T^{*}]; \dot{B}^{\gamma_{0}, \phi}_{p_{0}, \infty}\big) \cap C_{\alpha, \tilde{q}}\big((0, T^{*}]; L^{\tilde{r}}(\mathbb{R}^{d})\big)$ for some $T^{*} > 0$. Moreover, we obtain that
\begin{align*}
\big\|w\big\|_{L^{r}(\mathbb{R}^{d})} &\lesssim \big\|\mathcal{S}_{\alpha,\phi}(t)w_{0}\big\|_{L^{r}} + \left\|\int_{0}^{t}(t-\tau)^{\alpha-1} \mathcal{P}_{\alpha,\phi}(t-\tau)g(w(\tau,x))\,d\tau\right\|_{L^{r}} \\
&\lesssim t^{-\frac{\alpha}{q}}\big\|w_{0}\big\|_{\dot{B}^{\gamma_{0},\phi}_{p_{0},\infty}} + \int_{0}^{t}(t-\tau)^{\alpha-1-\frac{\alpha d}{2\delta}\left(\frac{1+\kappa}{\tilde{r}}-\frac{1}{r}\right)}\tau^{-\frac{\alpha(1+\kappa)}{\tilde{q}}} \big(\tau^{\frac{\alpha}{q}}\|w(\tau)\|_{L^{\tilde{r}}}\big)^{1+\kappa}\,d\tau \\
&\lesssim t^{-\frac{\alpha}{q}}\big\|w_{0}\big\|_{\dot{B}^{\gamma_{0},\phi}_{p_{0},\infty}} + t^{-\frac{\alpha}{q}} (T^{*})^{\alpha-\frac{\alpha\kappa d}{2\delta p}}\big\|w\big\|_{C_{\alpha, \tilde{q}}\big((0, T^{*}]; L^{\tilde{r}}(\mathbb{R}^{d})\big)},
\end{align*}
Furthermore, the continuity of $w \in C_{\alpha, q}\big((0, T^{*}]; L^{r}(\mathbb{R}^{d})\big)$ can be proved similarly. Hence, we conclude $w \in \mathcal{X}_{T^{*}}^{\alpha}$, which completes the proof.
\end{proof}
 In the following, we prove the asymptotic behavior of the E.q. \eqref{E.Q.T.F.S.E} as \( t \rightarrow \infty \).

\begin{theorem}\label{asymptotic behavior of mild solutions}
Let \( u \) and \( w \) be two global mild solutions of E.q. \eqref{E.Q.T.F.S.E} obtained from Theorem \ref{global well-posedbess mild solution}, corresponding to initial values \( u_{0} \) and \( w_{0} \), respectively. Then we have that
\begin{align*}
\lim_{t\rightarrow\infty}\big\|u-w\big\|_{\dot{B}^{\gamma_{0},\phi}_{p_{0},\infty}}+\lim_{t\rightarrow\infty}t^{\frac{\alpha}{q}}
\big\|u(t)-w(t)\big\|_{L^{r}}=0
\end{align*}
if and only if
\begin{align*}
\lim_{t\rightarrow\infty}\big\|\mathcal{S}_{\alpha,\phi}(t)(u_{0}-w_{0})\big\|_{\dot{B}^{\gamma_{0},\phi}_{p_{0},\infty}}
+\lim_{t\rightarrow\infty}t^{\frac{\alpha}{q}}\big\|\mathcal{S}_{\alpha,\phi}(t)(u_{0}-w_{0})\big\|_{L^{r}}=0
\end{align*}
\end{theorem}

\begin{proof}
\(\Leftarrow\)
We consider the norm of \( \|I(u,w)\|_{\mathcal{X}^{\alpha}} \), where
\[
I(u,w)(t)=\frac{1}{i}\int_{0}^{t}(t-\tau)^{\alpha-1}\mathcal{P}_{\alpha,\phi}(t-\tau)\big(g(u(\tau))-g(w(\tau))\big)d\tau.
\]
Note that Theorem \ref{global well-posedbess mild solution} holds, so we have \( \big\|u\big\|_{\mathcal{X}^{\alpha}}\leq R \) and \( \big\|w\big\|_{\mathcal{X}^{\alpha}}\leq R \), where \( R > 0 \) is a sufficiently small constant.
We obtain that
\begin{align*}
\big\|I(u,w)\big\|_{\dot{B}^{\gamma_{0},\phi}_{p_{0},\infty}} &\leq \big\|I(u,w)\big\|_{L^{p}(\mathbb{R}^{d})} \\
&\lesssim \int_{0}^{t}(t-\tau)^{\alpha-1-\frac{\alpha d}{2\delta}(\frac{1+\kappa}{r}-\frac{1}{p})}\big(\|u\|_{L^{r}}^{\kappa}+\|w\|_{L^{r}}^{\kappa}\big)\|u(\tau)-w(\tau)
\|_{L^{r}(\mathbb{R}^{d})}\,d\tau \\
&\lesssim R^{\kappa}\int_{0}^{t}(t-\tau)^{\alpha-1-\frac{\alpha d}{2\delta}(\frac{1+\kappa}{r}-\frac{1}{p})}\tau^{-\frac{\alpha\kappa}{q}}\big\|u(\tau)-w(\tau)\big\|_{L^{r}}\,d\tau,
\end{align*}
and
\begin{align*}
\big\|I(u,w)\big\|_{L^{r}(\mathbb{R}^{d})} \lesssim R^{\kappa}\int_{0}^{t}(t-\tau)^{\alpha-1-\frac{\alpha\kappa d}{2\delta r}}\tau^{-\frac{\alpha\kappa}{q}}\big\|u(\tau)-w(\tau)\big\|_{L^{r}}\,d\tau.
\end{align*}
We denote that
\[
J(t)=\big\|u(t)-w(t)\big\|_{\dot{B}^{\gamma_{0},\phi}_{p_{0},\infty}}+t^{\frac{\alpha}{q}}\big\|u(t)-w(t)\big\|_{L^{r}(\mathbb{R}^{d})},
\]
then we obtain that
\begin{align*}
&\big\|I(u,w)(t)\big\|_{\dot{B}^{\gamma_{0},\phi}_{p_{0},\infty}}+t^{\frac{\alpha}{q}}\big\|I(u,w)(t)\big\|_{L^{r}(\mathbb{R}^{d})} \\
&\lesssim R^{\kappa}\int_{0}^{t}(t-\tau)^{\alpha-1-\frac{\alpha d}{2\delta}(\frac{1+\kappa}{r}-\frac{1}{p})}\tau^{-\frac{\alpha(\kappa+1)}{q}}J(\tau)\,d\tau \\
&\quad +R^{\kappa}t^{\frac{\alpha}{q}}\int_{0}^{t}(t-\tau)^{\alpha-1-\frac{\alpha \kappa d}{2\delta r}}\tau^{-\frac{\alpha(\kappa+1)}{q}}J(\tau)\,d\tau \\
&\lesssim R^{\kappa}\int_{0}^{1}(1-\tau)^{\alpha-1-\frac{\alpha d}{2\delta}(\frac{1+\kappa}{r}-\frac{1}{p})}\tau^{-\frac{\alpha(\kappa+1)}{q}}J(t \tau)\,d\tau \\
&\quad +R^{\kappa}\int_{0}^{1}(1-\tau)^{\alpha-1-\frac{\alpha \kappa d}{2\delta r}}\tau^{-\frac{\alpha(\kappa+1)}{q}}J(t \tau)\,d\tau,
\end{align*}
hence we obtain that
\begin{align*}
\limsup_{t\rightarrow\infty}J(t) &\lesssim \limsup_{t\rightarrow\infty} \bigg[\big\|\mathcal{S}_{\alpha,\phi}(t)(u_{0}-w_{0})\big\|_{\dot{B}^{\gamma_{0},\phi}_{p_{0},\infty}}+t^{\frac{\alpha}{q}}
\big\|\mathcal{S}_{\alpha,\phi}(t)(u_{0}-w_{0})\big\|_{L^{r}}\bigg] \\
&\quad +\limsup_{t\rightarrow\infty} \bigg[\big\|I(u,w)(t)\big\|_{\dot{B}^{\gamma_{0},\phi}_{p_{0},\infty}}+t^{\frac{\alpha}{q}}\big\|I(u,w)(t)\big\|_{L^{r}(\mathbb{R}^{d})}\bigg] \\
&\lesssim R^{\kappa}\int_{0}^{1}(1-\tau)^{\alpha-1-\frac{\alpha d}{2\delta}(\frac{1+\kappa}{r}-\frac{1}{p})}\tau^{-\frac{\alpha(\kappa+1)}{q}}\limsup_{t\rightarrow\infty}J(t \tau)\,d\tau \\
&\quad +R^{\kappa}\int_{0}^{1}(1-\tau)^{\alpha-1-\frac{\alpha \kappa d}{2\delta r}}\tau^{-\frac{\alpha(\kappa+1)}{q}}\limsup_{t\rightarrow\infty}J(t \tau)\,d\tau \\
&\leq CR^{\kappa}\limsup_{t\rightarrow\infty}J(t),
\end{align*}
the constant \( R \) is small enough to ensure \( CR^{\kappa} < 1 \), hence this implies that \( \lim_{t\rightarrow\infty}J(t) = 0 \).

\(\Rightarrow\)
Note that
\[
\mathcal{S}_{\alpha,\phi}(t)\big(u_{0}-w_{0}\big)=u(t)-w(t)-I(u,w)(t),
\]
hence we obtain that
\begin{align*}
\lim_{t\rightarrow\infty}\big\|\mathcal{S}_{\alpha,\phi}(t)(u_{0}-w_{0})\big\|_{\dot{B}^{\gamma_{0},\phi}_{p_{0},\infty}}
+\lim_{t\rightarrow\infty}t^{\frac{\alpha}{q}}\big\|\mathcal{S}_{\alpha,\phi}(t)(u_{0}-w_{0})\big\|_{L^{r}}\leq (CR^{\kappa}+1)\limsup_{t\rightarrow\infty}J(t)=0.
\end{align*}
This completes the proof.
\end{proof}

Moreover, we also can construct the following local well-posedness of the mild solution to the E.q. \eqref{E.Q.T.F.S.E}.

\begin{theorem} \label{Local Well-posedness 1}

Let \((p, r, q)\) be a triple of admissible parameters and satisfy $p>d/2\delta$. If \textbf{Assumption H2} holds, and \(\kappa\) satisfies
\[
\frac{d(2\kappa+1)}{\delta p(\kappa+1)}<2,
\]
for \(w_{0} \in H^{s, \phi}_{p}(\mathbb{R}^{d})\), there exists \(T^{*} > 0\) such that the E.q. \eqref{E.Q.T.F.S.E} is locally well-posed in \(\mathcal{Y}^{\alpha}_{T^{*}}\), where
\[
\mathcal{Y}^{\alpha}_{T^{*}} = C\left([0, T^{*}]; H^{s, \phi}_{p}(\mathbb{R}^{d})\right) \cap C_{\alpha, q}\left((0, T^{*}]; H^{s, \phi}_{r}(\mathbb{R}^{d})\right),
\]
and \(s\) satisfies
\[
\max\left\{0, \frac{(\kappa+1)d}{\kappa\delta p}-\frac{2}{\kappa}\right\} < s \leq \frac{d\kappa}{\delta p(\kappa+1)}.
\]
Moreover, for any \( T' \in (0, T^{\ast}) \), there exists a neighborhood \( V \) of \( w_{0} \) in the space \( H^{s, \phi}_{p} \) such that the mapping \( \tilde{w}_{0} \mapsto \tilde{w} \) from \( V \) to \( \mathcal{Y}_{T'}^{\alpha} \) is Lipschitz continuous.
\end{theorem}
\begin{proof}

We define the operator \(\Theta\) as in Theorem \ref{Local exist of mild solution}, i.e., for any \(T > 0\),
\[
\Theta w(t, x) = \mathcal{S}_{\alpha, \phi}(t)w_{0}(x) + \frac{1}{i} \int_{0}^{t} (t - \tau)^{\alpha - 1} \mathcal{P}_{\alpha, \phi}(t - \tau) g(w(\tau, x)) \, d\tau, \quad t \in [0, T].
\]
Let \(s^{*} = (\kappa + 1)s - \frac{\kappa d}{\delta p}\). By Ribaud \cite[Theorem 1.2]{Ribaud}, we have
\[
g(w) \in H^{s^{*}, \phi}_{p}(\mathbb{R}^{d}), \quad \text{and} \quad \|g(w)\|_{H^{s^{*}, \phi}_{p}} \leq C \|w\|_{H^{s, \phi}_{p}}^{\kappa + 1}.
\]
By Lemma \ref{operator S}, we obtain that
\[
\|\mathcal{S}_{\alpha, \phi}(t)w_{0}\|_{H^{s, \phi}_{p}} \lesssim \|w_{0}\|_{H^{s, \phi}_{p}}, \quad \text{and} \quad t^{\frac{\alpha}{q}} \|\mathcal{S}_{\alpha, \phi}(t)w_{0}\|_{H^{s, \phi}_{r}} \lesssim \|w_{0}\|_{H^{s, \phi}_{p}},
\]
and
\[
\left\| \int_{0}^{t} (t - \tau)^{\alpha - 1} \mathcal{P}_{\alpha, \phi}(t - \tau) g(w(\tau, x)) \, d\tau \right\|_{H^{s, \phi}_{p}} \lesssim \int_{0}^{t} (t - \tau)^{\alpha - 1 - \frac{\alpha}{2\delta}(\delta(s - s^{*}))} \|g(w)(\tau)\|_{H^{s^{*}, \phi}_{p}} \, d\tau
\]
\[
\lesssim \int_{0}^{t} (t - \tau)^{\alpha - 1 - \frac{\alpha}{2\delta}(\delta(s - s^{*}))} \|w(\tau)\|_{H^{s, \phi}_{p}}^{\kappa + 1} \, d\tau \lesssim T^{\alpha - \frac{\alpha \kappa}{2} \left(\frac{d}{\delta p} - s\right)} \|w\|_{\mathcal{Y}_{T}^{\alpha}}^{\kappa + 1},
\]
\[
\left\| \int_{0}^{t} (t - \tau)^{\alpha - 1} \mathcal{P}_{\alpha, \phi}(t - \tau) g(w(\tau, x)) \, d\tau \right\|_{H^{s, \phi}_{r}} \lesssim \int_{0}^{t} (t - \tau)^{\alpha - 1 - \frac{\alpha}{2\delta}(\delta(s - s^{*}) - (\frac{d}{p} - \frac{d}{r}))} \|g(w)(\tau)\|_{H^{s^{*}, \phi}_{p}} \, d\tau
\]
\[
\lesssim \int_{0}^{t} (t - \tau)^{\alpha - 1 - \frac{\alpha}{2\delta}(\delta(s - s^{*}) - (\frac{d}{p} - \frac{d}{r}))} \|w(\tau)\|_{H^{s, \phi}_{p}}^{\kappa + 1} \, d\tau \lesssim t^{-\frac{\alpha}{q}} T^{\alpha - \frac{\alpha \kappa}{2} \left(\frac{d}{\delta p} - s\right)} \|w\|_{\mathcal{Y}_{T}^{\alpha}}^{\kappa + 1}.
\]
This implies that, similar to Theorem \ref{Local exist of mild solution}, we can get
\[
\|\Theta w\|_{\mathcal{Y}_{T}^{\alpha}} \lesssim \|w_{0}\|_{H^{s, \phi}_{p}} + T^{\alpha - \frac{\alpha \kappa}{2} \left(\frac{d}{\delta p} - s\right)} \|w\|_{\mathcal{Y}_{T}^{\alpha}}^{\kappa + 1}.
\]
Moreover, from Lemma \ref{embedding theorem}, we have that
$$
L^{a}(\mathbb{R}^{d})\hookrightarrow H^{s^{*},\phi}_{p}(\mathbb{R}^{d}),\text{ }\frac{1}{(\kappa+1)a}=\frac{1}{p}-\frac{\delta s}{d},
$$
hence we obtain that
\begin{align*}
&\bigg\|\int_{0}^{t}(t-\tau)^{\alpha-1}\mathcal{P}_{\alpha,\phi}(t-\tau)\big(g(u(\tau))-g(w(\tau))\big)\,d\tau
\bigg\|_{H^{s,\phi}_{p}(\mathbb{R}^{d})}\\
&\lesssim \int_{0}^{t}(t-\tau)^{\alpha-1-\frac{\alpha(s-s^{*})}{2}}\big\|g(u(\tau))-g(w(\tau))\big\|_{H^{s^{*},\phi}_{p}}\,d\tau\\
&\lesssim\int_{0}^{t}(t-\tau)^{\alpha-1-\frac{\alpha(s-s^{*})}{2}}\big\|g(u(\tau))-g(w(\tau))\big\|_{L^{a}}\,d\tau\\
&\lesssim\int_{0}^{t}(t-\tau)^{\alpha-1-\frac{\alpha(s-s^{*})}{2}}\big((\|u(\tau)\|_{L^{a(\kappa+1)}}^{\kappa}
+\|w(\tau)\|_{L^{a(\kappa+1)}}^{\kappa})\|u(\tau)-w(\tau)\|_{L^{a(\kappa+1)}}\big)\,d\tau\\
&\lesssim\int_{0}^{t}(t-\tau)^{\alpha-1-\frac{\alpha(s-s^{*})}{2}}\big((\|u(\tau)\|_{H^{s,\phi}_{p}}^{\kappa}
+\|w(\tau)\|_{H^{s,\phi}_{p}}^{\kappa})\|u(\tau)-w(\tau)\|_{H^{s,\phi}_{p}}\big)\,d\tau\\
&\lesssim T^{\alpha - \frac{\alpha \kappa}{2} \left(\frac{d}{\delta p} - s\right)}\left( \|u\|^{\kappa}_{\mathcal{Y}_{T}^{\alpha}} + \|w\|^{\kappa}_{\mathcal{Y}_{T}^{\alpha}} \right) \|u - w\|_{\mathcal{Y}_{T}^{\alpha}},
\end{align*}
and similar
\begin{align*}
&\bigg\|\int_{0}^{t}(t-\tau)^{\alpha-1}\mathcal{P}_{\alpha,\phi}(t-\tau)\big(g(u(\tau))-g(w(\tau))\big)\,d\tau
\bigg\|_{H^{s,\phi}_{r}(\mathbb{R}^{d})}\\
&\lesssim\int_{0}^{t}(t-\tau)^{\alpha - 1 - \frac{\alpha}{2\delta}(\delta(s - s^{*}) - (\frac{d}{p} - \frac{d}{r}))} \big\|g(u(\tau))-g(w(\tau))\big\|_{H^{s^{*},\phi}_{p}}\,d\tau\\
&\lesssim\int_{0}^{t}(t-\tau)^{\alpha - 1 - \frac{\alpha}{2\delta}(\delta(s - s^{*}) - (\frac{d}{p} - \frac{d}{r}))} \big((\|u(\tau)\|_{H^{s,\phi}_{p}}^{\kappa}
+\|w(\tau)\|_{H^{s,\phi}_{p}}^{\kappa})\|u(\tau)-w(\tau)\|_{H^{s,\phi}_{p}}\big)\,d\tau\\
&\lesssim t^{-\frac{\alpha}{q}}T^{\alpha - \frac{\alpha \kappa}{2} \left(\frac{d}{\delta p} - s\right)}\left( \|u\|^{\kappa}_{\mathcal{Y}_{T}^{\alpha}} + \|w\|^{\kappa}_{\mathcal{Y}_{T}^{\alpha}} \right) \|u - w\|_{\mathcal{Y}_{T}^{\alpha}},
\end{align*}
this implies that
\[
\|\Theta u - \Theta w\|_{\mathcal{Y}_{T}^{\alpha}} \lesssim T^{\alpha - \frac{\alpha \kappa}{2} \left(\frac{d}{\delta p} - s\right)} \left( \|u\|^{\kappa}_{\mathcal{Y}_{T}^{\alpha}} + \|w\|^{\kappa}_{\mathcal{Y}_{T}^{\alpha}} \right) \|u - w\|_{\mathcal{Y}_{T}^{\alpha}}.
\]
The verification that \(\Theta w(\cdot) \in C\left([0, T]; H^{s, \phi}_{p}(\mathbb{R}^{d})\right) \cap C_{\alpha, q}\left((0, T]; H^{s, \phi}_{r}(\mathbb{R}^{d})\right)\) is similar to Theorem \ref{Local exist of mild solution}, and we omit the details here.

Note that \(w_{0} \in H^{s, \phi}_{p}(\mathbb{R}^{d})\), so there exists a constant \(C\) such that \(\|w_{0}\|_{H^{s, \phi}_{p}} \leq \frac{R}{2C}\). Consider the closed ball \(\mathcal{B}_{R}\) in \(\mathcal{Y}_{T}^{\alpha}\), i.e.,
\[
\left\{ w \in \mathcal{Y}_{T}^{\alpha} : \|w\|_{\mathcal{Y}_{T}^{\alpha}} \lesssim R \right\}.
\]
Similar to Theorem \ref{Local exist of mild solution}, there exists a constant \(T^{*} > 0\) such that the E.q. \eqref{E.Q.T.F.S.E} has a unique mild solution in \(\mathcal{Y}^{\alpha}_{T^{*}}\).
This completes the proof.
\end{proof}
\begin{theorem}\label{Local Well-posedness 2}
Let $1 < p < \infty$, and suppose that \textbf{Assumption H2} holds. If the constant $\kappa$ satisfies
$$
\frac{d\kappa}{\delta p(\kappa+1)} \leq s < \min\left\{\frac{d}{p\delta}, 2\right\},
$$
for $w_{0} \in H^{s,\phi}_{p}(\mathbb{R}^{d})$, then there exists a constant $T^{*} > 0$ such that the E.q. \eqref{E.Q.T.F.S.E} has a unique mild solution in $C\left([0, T^{*}]; H^{s,\phi}_{p}(\mathbb{R}^{d})\right)$. Moreover, for any \( T' \in (0, T^{\ast}) \), there exists a neighborhood \( V \) of \( w_{0} \) in the space \( H^{s, \phi}_{p} \) such that the mapping \( \tilde{w}_{0} \mapsto \tilde{w} \) from \( V \) to $C\left([0, T']; H^{s,\phi}_{p}(\mathbb{R}^{d})\right)$ is Lipschitz continuous.
\end{theorem}
\begin{proof}

Similar to Theorem \ref{Local Well-posedness 1}, we define the operator $\Theta$ as in Theorem \ref{operator formula} for any $T > 0$, and we consider the following Banach space:
$$
\left\{ w \in C\left([0, T]; H^{s, \phi}_{p}(\mathbb{R}^{d})\right) : \ \|w\|_{L^{\infty}_{T}H^{s,\phi}_{p}} \lesssim R \right\}.
$$

Let \( s^{*} = (\kappa + 1)s - \frac{\kappa d}{\delta p} \). By Ribaud \cite[Theorem 1.2]{Ribaud}, we have
$$
g(w) \in H^{s^{*}, \phi}_{p}(\mathbb{R}^{d}), \quad \text{and} \quad \|g(w)\|_{H^{s^{*}, \phi}_{p}} \leq C \|w\|_{H^{s, \phi}_{p}}^{\kappa + 1}.
$$

By Lemma \ref{operator S}, we obtain that
$$
\|\mathcal{S}_{\alpha, \phi}(t)w_{0}\|_{H^{s, \phi}_{p}} \lesssim \|w_{0}\|_{H^{s, \phi}_{p}}
$$
and
\begin{align*}
&\bigg\| \int_{0}^{t} (t - \tau)^{\alpha - 1} \mathcal{P}_{\alpha, \phi}(t - \tau) g(w(\tau)) \, d\tau \bigg\|_{L^{\infty}_{T}H^{s,\phi}_{p}} \\
&\lesssim \int_{0}^{t} (t - \tau)^{\alpha - 1 - \frac{\alpha}{2\delta}(\delta(s - s^{*}))} \|g(w(\tau))\|_{H^{s^{*}, \phi}_{p}} \, d\tau \\
&\lesssim T^{\alpha - \frac{\alpha}{2}\left(\frac{\kappa d}{\delta p} - \kappa s\right)} \|w\|_{L^{\infty}_{T}H^{s,\phi}_{p}}^{\kappa + 1}.
\end{align*}

Note that \( H^{s, \phi}_{p}(\mathbb{R}^{d}) \hookrightarrow H^{\frac{d\kappa}{\delta p(\delta + 1)}, \phi}_{p}(\mathbb{R}^{d}) \hookrightarrow L^{(\kappa + 1)p}(\mathbb{R}^{d}) \), and combining H\"{o}lder's inequality, we also have that
\begin{align*}
&\bigg\|\int_{0}^{t}(t-\tau)^{\alpha-1}\mathcal{P}_{\alpha,\phi}(t-\tau)\big(g(u(\tau))-g(w(\tau))\big)\,d\tau
\bigg\|_{H^{s,\phi}_{p}(\mathbb{R}^{d})}\\
&\lesssim \int_{0}^{t}(t-\tau)^{\alpha-1-\frac{\alpha s}{2}}\big((\|u(\tau)\|^{\kappa}_{L^{(\kappa+1)p}}+\|w(\tau)\|^{\kappa}_{L^{(\kappa+1)p}})
(\|u(\tau)-w(\tau)\|_{L^{(\kappa+1)p}})\big)\,d\tau\\
&\lesssim \int_{0}^{t}(t-\tau)^{\alpha-1-\frac{\alpha s}{2}}\big(\|u(\tau)\|^{\kappa}_{H^{s,\phi}_{p}}
+\|w(\tau)\|^{\kappa}_{H^{s,\phi}_{p}}\big)
\big(\|u(\tau)-w(\tau)\|_{H^{s,\phi}_{p}}\big)\,d\tau\\
&\lesssim T^{\alpha-\frac{\alpha s}{2}} \big(\|u(\tau)-w(\tau)\|_{L^{\infty}_{T}H^{s,\phi}_{p}}\big)\big(\|u(\tau)\|^{\kappa}_{L^{\infty}_{T}H^{s,\phi}_{p}}
+\|w(\tau)\|^{\kappa}_{L^{\infty}_{T}H^{s,\phi}_{p}}\big)
.
\end{align*}
Similar the Theorem \ref{Local Well-posedness 1}, the E.q. \eqref{E.Q.T.F.S.E} possesses a unique mild solution on $C\big([0,T^{*}];H^{s,\phi}_{p}(\mathbb{R}^{d})\big)$ for some $T^{*} > 0$.
This completes the proof.
\end{proof}

\section{Conclusion}
Based on Kim's study \cite[Adv. Math.]{Kim2} of the operator $\phi(-\Delta)$, this paper investigates the time-space fractional Schr\"{o}dinger E.q. \eqref{E.Q.T.F.S.E} driven by $\phi(-\Delta)$-type operators. This is a generalization of the equation \eqref{E.Q.T.F.S.E.1} driven by the fractional Laplacian operator $(-\Delta)^{\frac{\beta}{2}}$ studied by Su \cite[J. Math. Anal. Appl.]{Su} and others. Since the Bernstein function $\phi$ does not possess homogeneity, the method in \cite{Su} fails for E.q. \eqref{E.Q.T.F.S.E}. To overcome this difficulty, we first establish a fractional Gagliardo-Nirenberg inequality in the generalized Triebel-Lizorkin space $F^{s,\phi}_{p,q}$, and combine it with H\"{o}rmander multiplier theory, asymptotic properties of the Mittag-Leffler function, and properties of Bernstein functions. We then establish the $L^{p}-L^{r}$ estimate for the operators $\mathcal{S}_{\alpha,\phi}(t)$ and $\mathcal{P}_{\alpha,\phi}(t)$. In particular, when $\phi(-\Delta)=(-\Delta)^{\frac{\beta}{2}}$, it recovers the $L^{p}-L^{r}$ estimate established in \cite{Su}. By using real interpolation, Sobolev embedding, and other harmonic analysis methods, we prove the local/global well-posedness of mild solutions to  \eqref{E.Q.T.F.S.E} in some Banach space. This result complements Su's \cite{Su} result, and the method is entirely different.

Moreover, if the linear part considers the model proposed by Naber \cite{Naber}, then $i$ is replaced by $i^{\alpha}$. In this case, for $z=(-it)^{\alpha}\phi(|\xi|^{2})$, we observe that $|\arg(z)|=\alpha\pi/2$. From Proposition \ref{X.Z.of Mitag function}, the solution operators $\mathcal{S}_{\alpha,\phi}(t)$ and $\mathcal{P}_{\alpha,\phi}(t)$ will individually contain the unitary operator terms $\exp(-it(\phi(-\Delta))^{\frac{1}{\alpha}})$ and $(\phi(-\Delta))^{\frac{1-\alpha}{\alpha}}\exp(-it(\phi(-\Delta))^{\frac{1}{\alpha}})$. In this case, since the symbol does not exhibit decay, H\"{o}rmander multiplier theory will completely fail, and new methods need to be developed. In particular, when $\phi(-\Delta)=(-\Delta)^{\frac{\beta}{2}}$, we get the operator $\exp(-itD^{\delta})$ and $\exp(-itD^{\delta})$ and $D^{\delta-\beta}\exp(-itD^{\delta})$, where $\delta=\beta/\alpha$. Inspired by \cite{Kenig, Grande, Lee}, we also study the following well-posedness problem for the time-space fractional Schr\"{o}dinger equation \cite{YangZhou2026}:
\begin{align}
\begin{cases}
    i^{\alpha}\partial_{t}^{\alpha}w(t,x) = (-\Delta)^{\frac{\beta}{2}}w(t,x) + g(w(t,x)) & \text{in } (0,\infty) \times \mathbb{R}^{d}, \\
    w(0,x) = w_{0}(x) & \text{in } \mathbb{R}^{d},
\end{cases}
\end{align}

\noindent{\bf Declaration of competing interest}\\
The authors declare that they have no competing interests.\\
\noindent{\bf Data availability}\\
No data was used for the research described in the article.\\
\noindent{\bf Acknowledgements}\\
This work was supported by National Natural Science Foundation of China (12471172).

\end{document}